\documentclass[11pt,a4paper]{article}
\usepackage[utf8x]{inputenc}
\usepackage{amsmath}
\usepackage{amsfonts}
\usepackage{amsthm}
\usepackage{amssymb}
\usepackage{graphicx}
\usepackage[margin=1in]{geometry}
\usepackage{hyperref}

\usepackage{mathtools}
\usepackage{apptools}
\usepackage{xcolor}
\usepackage{enumitem}
\usepackage{mathabx}
\usepackage{cases}
\usepackage{stmaryrd}
\usepackage[framemethod=tikz]{mdframed}

\newtheorem{thm}{Theorem}
\newtheorem{lem}[thm]{Lemma}
\newtheorem{prop}[thm]{Proposition}

\newtheorem{assume}{Assumption}

\theoremstyle{definition}

\newtheorem{rmk}[thm]{Remark}

\setlist[enumerate]{label=$\rm{(\roman*)}$,leftmargin=\parindent}
\numberwithin{equation}{section}
\numberwithin{thm}{section}
\numberwithin{hypo}{section}
\numberwithin{table}{section}
\numberwithin{figure}{section}

\newcommand{\sR}{\mathbb{R}}
\newcommand{\sL}{\mathbb{L}}
\newcommand{\sX}{\mathcal{X}}
\newcommand{\sY}{\mathcal{Y}}

\newcommand{\sol}{\mathbb{S}}
\newcommand{\Fea}{\mathbb{F}}

\newcommand{\mysum}{\displaystyle \sum\limits}

\newcommand{\bO}{\mathcal{O}}
\newcommand{\sB}{\mathbb{B}}

\newcommand{\E}{\mathcal{E}}
\newcommand{\G}{\mathcal{G}}

\newcommand{\Lag}{\mathcal{L}}
\newcommand{\Lb}{\mathcal{L}_{\beta}}

\newcommand{\TL}{\mathcal{T}_{\Lag}}

\newcommand{\bz}{\widetilde{z}}
\newcommand{\bt}{\widetilde{t}}
\newcommand{\bu}{\widetilde{u}}

\newcommand{\bzeta}{\widetilde{\zeta}}
\newcommand{\brho}{\widetilde{\rho}}

\newcommand{\tx}{\widetilde{x}}

\newcommand{\tlambda}{\widetilde{\lambda}}

\newcommand{\Csup}{C_{0}}
\newcommand{\Cbnd}{C_{4}}
\newcommand{\CLag}{C_{1}}
\newcommand{\Cv}{C_{2}}
\newcommand{\Cite}{C_{3}}
\newcommand{\Cint}{C_{5}}
\newcommand{\Cgra}{C_{6}}
\newcommand{\Cfea}{C_{7}}
\newcommand{\Cacc}{C_{8}}
\newcommand{\Czms}{C_{9}}
\newcommand{\Tm}{T_{\max}}

\title{Improved convergence rates and trajectory convergence for primal-dual dynamical systems with vanishing damping}
\author{Radu Ioan Bo\c{t}\footnote{Faculty of Mathematics, University of Vienna, Oskar-Morgenstern-Platz 1, 1090 Vienna, Austria, e-mail: \url{radu.bot@univie.ac.at}.}
\and Dang-Khoa Nguyen\footnote{Faculty of Mathematics, University of Vienna, Oskar-Morgenstern-Platz 1, 1090 Vienna, Austria, e-mail: \url{dang-khoa.nguyen@univie.ac.at}}}

\begin{document}
	
\maketitle	
	
\begin{abstract}
In this work, we approach the minimization of a continuously differentiable convex function under linear equality constraints by a second-order dynamical system with asymptotically vanishing damping term. The system is formulated in terms of the augmented Lagrangian associated to the minimization problem. We show fast convergence of the primal-dual gap, the feasibility measure, and the objective function value along the generated trajectories. In case the objective function has Lipschitz continuous gradient, we show that the primal-dual trajectory asymptotically weakly converges to a primal-dual optimal solution of the underlying minimization problem. To the best of our knowledge, this is the first result which guarantees the convergence of the trajectory generated by a primal-dual dynamical system with asymptotic vanishing damping. Moreover, we will rediscover in case of the unconstrained minimization of a convex differentiable function with Lipschitz continuous gradient all convergence statements obtained in the literature for Nesterov's accelerated gradient method.
\end{abstract}	

\vspace{1ex}

	
\section{Introduction}	
	
\subsection{Problem statement and motivation}	
	
In this paper we will deal with the optimization problem
\begin{equation}
	\label{intro:pb}
	\begin{array}{rl}
		\min & f \left( x \right), \\
		\textrm{subject to} 	& Ax = b
	\end{array}
\end{equation}
where
\begin{equation}
	\label{intro:hypo}
	\begin{cases}
		\sX, \sY \textrm{ are real Hilbert spaces}; \\
		f \colon \sX \to \sR \textrm{ is a continuously differentiable convex function}; \\
		A \colon \sX \to \sY \textrm{ is a continuous linear operator and } b \in \sY; \\
		\textrm{the set } \mathbb{S} \textrm{ of primal-dual optimal solutions of } \eqref{intro:pb} \textrm{ is assumed to be  nonempty}.
	\end{cases}
\end{equation}

Problems of type \eqref{intro:pb} underlie many important applications in various areas, such as image recovery \cite{Goldstein-ODonoghue-Setzer-Baraniuk}, machine learning \cite{Boyd-et.al,Lin-Li-Fang}, the energy dispatch of power grids \cite{Yi-Hong-Liu:15,Yi-Hong-Liu:16}, distributed optimization \cite{Madan-Lall,Zeng-Yi-Hong-Xie} and network optimization \cite{Shi-Johansson,Zeng-Lei-Chen}.

The object of our investigations will be a second-order dynamical system with asymptotic vanishing damping term associated with the optimization problem \eqref{intro:pb} and formulated in terms of its augmented Lagrangian. Our main aim is to study the asymptotic behaviour of the generated trajectories to a primal-dual optimal solution as well as to derive fast rates of convergence for the primal-dual gap, the feasibility measure, and the objective function value along these.

The interplay between continuous-time dissipative dynamical systems and numerical algorithms for solving optimization problems has been subject of an intense research activity. It is well-known for unconstrained optimization problems that damped inertial dynamics are a natural way to accelerate these systems. In line with the seminal work of Polyak on the heavy ball method with friction \cite{Polyak,Polyak:book}, the first studies by Alvarez and Attouch focused on inertial dynamics with fixed viscous damping coefficient \cite{Alvarez, Alvarez-Attouch-Bolte-Redont,Attouch-Goudou-Redont}. A decisive step was taken by Su, Boyd and Cand\`{e}s in \cite{Su-Boyd-Candes}, where, for the minimization of a continuously differentiable convex function $f: \sX \rightarrow \sR$, the following inertial dynamics with an asymptotically vanishing damping coefficient has been considered
\begin{equation}
\tag{$\mathrm{AVD}$}
\label{ds:AVD}
\ddot{x} \left( t \right) + \dfrac{\alpha}{t} \dot{x} \left( t \right) + \nabla f \left( x \left( t \right) \right)   = 0.
\end{equation}
The terminology asymptotic vanishing damping (AVD) refers to the specific characteristic of the damping coefficient $\dfrac{\alpha}{t}$ to vanish in a controlled manner, neither too fast nor too slowly, as $t$ goes to infinity. In particular, in the case $\alpha = 3$, this dynamical system can be seen as the continuous limit of Nesterov’s accelerated gradient algorithm \cite{Nesterov:83,Nesterov:13,FISTA}. In the last years, the community paid a lot of attention to the topic of inertial dynamics \cite{Attouch-Cabot:17,Attouch-Cabot-Chbani-Riahi,Attouch-Chbani-Peypouquet-Redont,Attouch-Chbani-Riahi:ESAIM,Attouch-Peypouquet-Redont,Bot-Csetnek,Bot-Csetnek-Laszlo, Bot-Csetnek-Laszlo:18,Jendoubi-May,May}, as well as of their discrete counterparts \cite{Apidopoulos-Aujol-Dossal,Attouch-Cabot:18,Attouch-Chbani-Fadili-Riahi:20,Attouch-Peypouquet:16,Aujol-Dossal,Chambolle-Dossal}, to name only a few.

The augmented Lagrangian Method (ALM) \cite{Rockafellar:MOOR} (for linearly constrained problems), the Alternating Direction Method of Multipliers (ADMM) \cite{Gabay-Meicer,Boyd-et.al} (for problems with separable objectives and block variables linearly coupled in the constraints) and some of their variants have proved to be very suitable when solving large-scale structured convex optimization problems. Since the primal-dual systems of optimality conditions to be solved can be equivalently formulated as monotone inclusion problems, see \cite{Rockafellar:70,Rockafellar:MOOR,Rockafellar:SICON}, the above-mentioned methods are intimately linked with numerical algorithms designed to find a zero of a maximally monotone operator. This close connection has been used in recent works addressing the acceleration of ADMM/ALM methods via inertial dynamics.  In \cite{Bot-Csetnek:iADMM}, for instance, an inertial ADMM numerical algorithm has been proposed originating in the inertial version of the Douglas-Rachford splitting method for monotone inclusion problems introduced in \cite{Bot-Csetnek-Hendrich}. Recently, Attouch has proposed in \cite{Attouch} an inertial proximal ADMM algorithm, relying on the general scheme from \cite{Attouch-Peypouquet:19} designed to solve general monotone inclusions and in lines with  \cite{Attouch-Soueycatt}, and investigated its fast convergence properties for certain combinations of the viscosity and the proximal parameters. However, the inertial proximal ADMM algorithm fails to be a full splitting method.

Continuous-time approaches for structured convex minimization problems formulated in the spirit of the full splitting paradigm have been recently addressed in \cite{Bot-Csetnek-Laszlo:20} and, closely connected to our approach, in \cite{Zeng-Lei-Chen,He-Hu-Fang,Attouch-Chbani-Fadili-Riahi}, to which we will have a closer look in Subsection \ref{pdvanishing}.

\subsection{Our contributions} 

For a primal-dual dynamical system with asymptotically vanishing damping term associated to the augmented Lagrangian formulation of  \eqref{intro:pb} we will show fast convergence for the primal-dual gap, the feasibility measure, and the objective function value along the generated trajectories, and, consequently, improve existing results in the literature. We will prove the existence and uniqueness of the trajectories as global twice continuously differentiable solutions of the dynamical system provided the gradient of the objective function is Lipschitz continuous. In the same setting, we will also prove that the primal-dual trajectory asymptotically weakly converges to a primal-dual optimal solution of  \eqref{intro:pb}, which is the first result of this type in the literature addressing such dynamical systems.

Last but not least, we will show how the asymptotic analysis and the obtained results can be straightforwardly transferred to continuous-time methods with vanishing damping terms approaching optimization problems with separable objectives and block variables linearly coupled in the constraints. Moreover, we will rediscover in case of the unconstrained minimization of a convex differentiable function with Lipschitz continuous gradient all convergence statements obtained in the literature for Nesterov's accelerated gradient method introduced in \cite{Su-Boyd-Candes, Attouch-Chbani-Peypouquet-Redont}.

\subsection{Notations and a preliminary result}

For both Hilbert spaces $\sX$ and $\sY$, the Euclidean inner product and the associated norm will be denoted by $\left\langle \cdot , \cdot \right\rangle$ and $\left\lVert \cdot \right\rVert$, respectively. The Cartesian product $\sX \times \sY$ will be endowed with the inner product and the associated norm defined for $\left( x , \lambda \right) , \left( z , \mu \right) \in \sX \times \sY$ as
\begin{equation*}
\left\langle \left( x , \lambda \right) , \left( z , \mu \right) \right\rangle
= \left\langle x , z \right\rangle + \left\langle \lambda , \mu \right\rangle
\qquad \textrm{ and } \qquad
\left\lVert \left( x , \lambda \right) \right\rVert
= \sqrt{\left\lVert x \right\rVert ^{2} + \left\lVert \lambda \right\rVert ^{2}},
\end{equation*}
respectively. The closed ball centered at $x \in \sX$ with radius $\varepsilon > 0$ will be denoted by $\sB \left( x ; \varepsilon \right) := \left\lbrace y \in \sX \colon \left\lVert x - y \right\rVert \leq \varepsilon \right\rbrace$.

Let $f \colon \sX \to \sR$ be a continuously differentiable convex function such that $\nabla f$ is $\ell-$Lipschitz continuous.
For every $x, y \in \sX$ it holds (see \cite[Theorem 2.1.5]{Nesterov:book})
\begin{equation}
\label{pre:f-bound}
0 \leq \dfrac{1}{2 \ell} \left\lVert \nabla f \left( x \right) - \nabla f \left( y \right) \right\rVert ^{2} \leq f \left( x \right) - f \left( y \right) - \left\langle \nabla f \left( y \right) , x - y \right\rangle \leq \dfrac{\ell}{2} \left\lVert x - y \right\rVert ^{2} .
\end{equation}

\section{The primal-dual dynamical approach with vanishing damping}
\label{sec:system}

\subsection{Augmented Lagrangian formulation}

Consider the saddle point problem
\begin{equation}
\label{intro:sd}
\min_{x \in \sX} \max_{\lambda \in \sY} \Lag \left( x , \lambda \right)
\end{equation}
associated to problem \eqref{intro:pb}, where $\Lag \colon \sX \times \sY \to \sR$ denotes the Lagrangian function
\begin{equation*}
\Lag \left( x , \lambda \right) := f \left( x \right) + \left\langle \lambda , Ax - b \right\rangle .
\end{equation*}
Under the assumptions \eqref{intro:hypo}, $\Lag$ is convex with respect to $x \in \sX$ and affine with respect to $\lambda \in \sY$. A pair $\left( x_{*} , \lambda_{*} \right) \in \sX \times \sY$ is said to be a saddle point of the Lagrangian function $\Lag$ if for every $\left( x , \lambda \right) \in \sX \times \sY$
\begin{equation}
\label{intro:sadde-point}
\Lag \left( x_{*} , \lambda \right) \leq \Lag \left( x_{*} , \lambda_{*} \right) \leq \Lag \left( x , \lambda_{*} \right).
\end{equation}
If $\left( x_{*} , \lambda_{*} \right) \in \sX \times \sY$ is a saddle point of $\Lag$ then $x_{*} \in \sX$ is an optimal solution of \eqref{intro:pb}, and $\lambda_{*} \in \sY$ is an optimal solution of its Lagrange dual problem. If $x_{*} \in \sX$ is an optimal solution of \eqref{intro:pb} and a suitable constraint qualification is fulfilled, then there exists an optimal solution $\lambda_{*} \in \sY$ of the Lagrange dual problem such that $\left( x_{*} , \lambda_{*} \right) \in \sX \times \sY$ is a saddle point of $\Lag$. For details and insights into the topic of constraint qualifications for convex duality we refer to \cite{Bauschke-Combettes:book,Bot:book}.

The set of saddle points of $\Lag$, called also primal-dual optimal solutions of \eqref{intro:pb}, will be denoted by $\sol$ and, as stated in the assumptions, it will be assumed to be nonempty.  The set of feasible points of \eqref{intro:pb} will be denoted by $\Fea := \left\lbrace x \in \sX \colon Ax = b \right\rbrace$ and the optimal objective value of \eqref{intro:pb}  by $f_*$.

The system of primal-dual optimality conditions for \eqref{intro:pb} reads
\begin{equation}
\label{intro:opt-Lag}
\left( x_{*} , \lambda_{*} \right) \in \sol
\Leftrightarrow \begin{cases}
\nabla_{x} \Lag \left( x_{*} , \lambda_{*} \right) 			& = 0 \\
\nabla_{\lambda} \Lag \left( x_{*} , \lambda_{*} \right) 	& = 0
\end{cases} \Leftrightarrow \begin{cases}
\nabla f \left( x_{*} \right) + A^* \lambda_{*} 	& = 0 \\
Ax_{*} - b 													& = 0
\end{cases},
\end{equation}
where $A^* : \sY \rightarrow \sX$ denotes the adjoint operator of $A$.

For $\beta \geq 0$, we consider also the augmented Lagrangian $\Lb \colon \sX \times \sY \to \sR$ associated with \eqref{intro:pb} 
\begin{equation}
	\label{intro:aug-Lag}
\Lb \left( x , \lambda \right) := \Lag \left( x , \lambda \right) + \dfrac{\beta}{2} \left\lVert Ax - b \right\rVert ^{2} = f \left( x \right) + \left\langle \lambda , Ax - b \right\rangle + \dfrac{\beta}{2} \left\lVert Ax - b \right\rVert ^{2} .
\end{equation}
For every $(x, \lambda) \in \Fea \times \sY$ it holds
\begin{equation}
\label{intro:Fea:eq}
f \left( x \right) = \Lb \left( x , \lambda \right) = \Lag \left( x , \lambda \right) .
\end{equation}
If $\left( x_{*} , \lambda_{*} \right) \in \sol$, then we have for every $\left( x , \lambda \right) \in \sX \times \sY$
\begin{equation*}
\Lag \left( x_{*} , \lambda \right) = \Lb \left( x_{*} , \lambda \right) 
\leq \Lag \left( x_{*} , \lambda_{*} \right) = \Lb \left( x_{*} , \lambda_{*} \right) 
\leq \Lag \left( x , \lambda_{*} \right) \leq \Lb \left( x , \lambda_{*} \right).
\end{equation*}
In addition,
\begin{equation}
\label{intro:opt-Lb}
\left( x_{*} , \lambda_{*} \right) \in \sol
\Leftrightarrow \begin{cases}
\nabla_{x} \Lb \left( x_{*} , \lambda_{*} \right) 			& = 0 \\
\nabla_{\lambda} \Lb \left( x_{*} , \lambda_{*} \right) 	& = 0
\end{cases} \Leftrightarrow \begin{cases}
\nabla f \left( x_{*} \right) + A^* \lambda_{*} 	& = 0 \\
Ax_{*} - b 													& = 0
\end{cases} .
\end{equation}

\subsection{Associated monotone inclusion problem}

The optimality system \eqref{intro:opt-Lag} can be equivalently written  as
\begin{equation}\label{moninclusion}
\TL \left( x_{*} , \lambda_{*} \right) = 0 ,
\end{equation}
where
\begin{equation}\label{TL}
\TL \colon \sX \times \sY \to \sX \times \sY, \quad \TL \left( x , \lambda \right)
= \begin{pmatrix}
\nabla_{x} \Lag \left( x , \lambda \right) \\ - \nabla_{\lambda} \Lag \left( x , \lambda \right)
\end{pmatrix}
= \begin{pmatrix}
\nabla f \left( x \right) + A^* \lambda \\ b-Ax
\end{pmatrix},
\end{equation}
is the maximally monotone operator associated with the convex-concave function $\Lag$. Indeed, it is immediate to verify that $\TL$ is monotone. Since it is also continuous, it is maximally monotone (see, for instance, \cite[Corollary 20.28]{Bauschke-Combettes:book}). Therefore $\sol$ can be interpreted as the set of zeros of the maximally
monotone operator $\TL$, which means that it is a closed convex subset of $\sX \times \sY$ (see, for instance, \cite[Proposition 23.39]{Bauschke-Combettes:book}).

Applying the fast continuous-time approaches recently proposed in \cite{Attouch-Peypouquet:19,Attouch} to the solving of \eqref{moninclusion} would require the use of the Moreau-Yosida approximation of the operator $\TL$, for which in general no close formula is available. The resulting dynamical system would therefore not be formulated in the spirit of the full splitting algorithm, which is undesirable from the point of view of numerical computations.

\subsection{The primal-dual dynamical system with vanishing damping}\label{pdvanishing}

The dynamical system which we associate to \eqref{intro:pb} and investigate in this paper reads
\begin{mdframed}	
	\begin{equation}
		\tag{$\mathrm{PD}$-$\mathrm{AVD}$}
		\label{ds:PD-AVD}
		\begin{dcases}
			\ddot{x} \left( t \right) + \dfrac{\alpha}{t} \dot{x} \left( t \right) + \nabla_{x} \Lb \Bigl( x \left( t \right) , \lambda \left( t \right) + \theta t \dot{\lambda} \left( t \right) \Bigr)  		& = 0 \\
			\ddot{\lambda} \left( t \right) + \dfrac{\alpha}{t} \dot{\lambda} \left( t \right) - \nabla_{\lambda} \Lb \Bigl( x \left( t \right) + \theta t \dot{x} \left( t \right) , \lambda \left( t \right) \Bigr) 	& = 0 \\
\Bigl( x \left( t_{0} \right) , \lambda \left( t_{0} \right) \Bigr) 			= \Bigl( x_{0} , \lambda_{0} \Bigr) \textrm{ and }
\Bigl( \dot{x} \left( t_{0} \right) , \dot{\lambda} \left( t_{0} \right) \Bigr) = \Bigl( \dot{x}_{0} , \dot{\lambda}_{0} \Bigr)
		\end{dcases},
	\end{equation}
\end{mdframed}
where $t_0 >0$, $\alpha \geq 3$, $\beta \geq 0$, $\theta >0$ and $(x_0, \lambda_0), (\dot{x}_{0} , \dot{\lambda}_{0}) \in \sX \times \sY$.

Our system is a particular case of the Temporally Rescaled Inertial Augmented Lagrangian System (TRIALS) proposed by Attouch, Chbani, Fadili and Riahi in \cite{Attouch-Chbani-Fadili-Riahi}
\begin{equation}
\tag{$\mathrm{TRIALS}$}
\label{ds:TRIALS}
\begin{dcases}
\ddot{x} \left( t \right) + \gamma \left( t \right) \dot{x} \left( t \right) + b \left( t \right) \nabla_{x} \Lb \Bigl( x \left( t \right) , \lambda \left( t \right) + \theta \left( t \right) \dot{\lambda} \left( t \right) \Bigr)  		& = 0 \\
\ddot{\lambda} \left( t \right) + \gamma \left( t \right) \dot{\lambda} \left( t \right) - b \left( t \right) \nabla_{\lambda} \Lb \Bigl( x \left( t \right) + \theta \left( t \right) \dot{x} \left( t \right) , \lambda \left( t \right) \Bigr) 	& = 0 
\end{dcases},
\end{equation}
where $\gamma, \theta, b \colon \left[ t_{0} , + \infty \right) \to (0,+\infty)$ are continuously differentiable functions. 
The case when $b$ is identically $1$ was also studied by He, Hu and Fang in \cite{He-Hu-Fang}. In \cite{Attouch-Chbani-Fadili-Riahi,He-Hu-Fang} the authors have actually investigated the minimization of the sum of two separable functions with the block variables linked by linear constraints, however, we will see in the next subsection that our analysis can be easily extended to this setting. 

The viscous damping function $\gamma(\cdot)$ is vital in achieving fast convergence and its role has been already well-understood in unconstrained minimization \cite{Attouch-Cabot:17,Attouch-Cabot-Chbani-Riahi,Jendoubi-May} (see also \cite{Attouch-Chbani-Peypouquet-Redont,Attouch-Chbani-Riahi:ESAIM,May} for the case when $\gamma \left( t \right) := \dfrac{\alpha}{t}$). The role of the extrapolation function $\theta(\cdot)$ is to induce more flexibility in the dynamical system and in the associated discrete schemes, as it has been recently noticed in \cite{Attouch-Chbani-Fadili-Riahi,Attouch-Chbani-Riahi:Opti,He-Hu-Fang,Zeng-Lei-Chen}.  The time scaling function $b(\cdot)$ has the role to further improve the rates of convergence of the objective function value along the trajectory, as it was noticed in the context of uncostrained minimization problems in \cite{Attouch-Chbani-Fadili-Riahi:20,Attouch-Chbani-Riahi:SIOPT,Attouch-Chbani-Riahi:PAFA} and of linearly constrained minimization problems in  \cite{Attouch-Balhag-Chbani-Riahi}.

The dynamical system \eqref{ds:PD-AVD} is \eqref{ds:TRIALS} for
\begin{equation*}
\gamma(t) := \dfrac{\alpha}{t} , \qquad \theta(t) := \theta t \qquad \textrm{ and } \qquad b(t) := 1 \qquad \forall t \geq t_{0} ,
\end{equation*}
where $\alpha \geq 3$ and $\theta > 0$. A setting which is closely related to ours can be found in the work \cite{Zeng-Lei-Chen} of Zeng, Lei and Chen.  However, when compared to \cite{Attouch-Chbani-Fadili-Riahi,He-Hu-Fang,Zeng-Lei-Chen}, we provide improved convergence rates and also prove weak convergence of the trajectories to a primal-dual optimal  solution. We also expect that our analysis can be adapted to the more general system  \eqref{ds:TRIALS}, though, we prefer the particular setting of \eqref{ds:PD-AVD}, in order to keep the presentation more simple and easier to follow.

Since our system is a particular instance of \eqref{ds:TRIALS}, we could have relied on the results showing the existence and uniqueness of a strong global solution from \cite{Attouch-Chbani-Fadili-Riahi}. We will prove instead the existence and uniqueness of the trajectories as global twice continuously differentiable solutions of \eqref{ds:PD-AVD}, provided $\nabla f$ is Lipschitz continuous.

Replacing the expressions of the partial gradients of $\Lb$ into the system leads to the following formulation for \eqref{ds:PD-AVD}
\begin{equation}
\label{ds:Cauchy}
\begin{dcases}
\ddot{x} \left( t \right) + \dfrac{\alpha}{t} \dot{x} \left( t \right) + \nabla f \left( x \left( t \right) \right) + A^* \left( \lambda \left( t \right) + \theta t \dot{\lambda} \left( t \right) \right) + \beta A^* \Bigl( Ax \left( t \right) - b \Bigr)		& = 0 \\
\ddot{\lambda} \left( t \right) + \dfrac{\alpha}{t} \dot{\lambda} \left( t \right) - \Bigl( A \bigl( x \left( t \right) + \theta t \dot{x} \left( t \right) \bigr) - b \Bigr) 	& = 0 \\
\Bigl( x \left( t_{0} \right) , \lambda \left( t_{0} \right) \Bigr) 			= \Bigl( x_{0} , \lambda_{0} \Bigr) \textrm{ and }
\Bigl( \dot{x} \left( t_{0} \right) , \dot{\lambda} \left( t_{0} \right) \Bigr) = \Bigl( \dot{x}_{0} , \dot{\lambda}_{0} \Bigr)
\end{dcases}.
\end{equation}

\subsection{Extension to multi-block optimization problems}
\label{subsec:multi-block}

For $m \geq 2$ a positive integer,  we consider the minimization of a separable objective function with respect to linearly coupled block variables 
\begin{equation}
\label{intro:pb-mb}
\begin{array}{rl}
\min	& f_{1} \left( x_{1} \right) + \cdots + f_{m} \left( x_{m} \right), \\
\qquad \ \textrm{subject to} & A_{1} x_{1} + \cdots + A_{m} x_{m} = b
\end{array}
\end{equation}
where
\begin{equation}
\label{intro:hypo-mb}
\begin{cases}
\sX_{i}, i=1, ..., m, \ \textrm{and} \ \sY \ \textrm{are real Hilbert spaces}; \\
f_{i} \colon \sX_{i} \to \sR, i=1, ..., m, \ \textrm{ are continuously differentiable convex functions}; \\
A_{i} \colon \sX_{i} \to \sY, i=1, ..., m, \textrm{ are continuous linear operators and } b \in \sY; \\
\textrm{the set of primal-dual optimal solutions of} \ \eqref{intro:pb-mb} \ \textrm{ is nonempty}.
\end{cases}
\end{equation}

Let $\sX := \sX_{1} \times \cdots \times \sX_{m}$ be the Cartesian product of the real Hilbert spaces $\sX_i, i=1, ..., m$, endowed with inner product and associated norm defined for $x := \left( x_{1} , \cdots , x_{m} \right) , z := \left( z_{1} , \cdots , z_{m} \right) \in  \sX$ as
\begin{equation*}
\left\langle x , z \right\rangle  = \mysum_{i=1}^{m} \left\langle x_{i} , z_{i} \right\rangle
\qquad \textrm{ and } \qquad
\left\lVert x \right\rVert
= \sqrt{\mysum_{i=1}^{m} \left\lVert x_{i} \right\rVert ^{2}} .
\end{equation*}
The multi-block optimization problem \eqref{intro:pb-mb} can be equivalently written as \eqref{intro:pb}, for the separable objective function
\begin{equation*}
f \colon \sX \to \sR, \quad f \left( x \right) = f \left( x_{1} , \cdots , x_{m} \right) := \mysum_{i=1}^{m} f_{i} \left( x_{i} \right) ,
\end{equation*}
and the continuous linear operator
\begin{equation*}
A \colon \sX \to \sY, \quad Ax = A \left( x_{1} , \cdots , x_{m} \right) = \mysum_{i=1}^{m} A_{i}x_{i}.
\end{equation*}
Since
\begin{equation*}
\nabla f \left( x \right)
=  \begin{pmatrix}
\nabla f_{1} \left( x_{1} \right) \\ \vdots \\ \nabla f_{m} \left( x_{m} \right) 
\end{pmatrix} \quad \mbox{for} \ x=(x_1, ..., x_m),
\end{equation*}
and 
\begin{equation*}
A^* \colon \sY \to \sX = \sX_{1} \times \cdots \times \sX_{m},  \quad A^* \lambda = \begin{pmatrix}
A_{1}^{*} \lambda \\ \vdots \\ A_{m}^{*} \lambda
\end{pmatrix},
\end{equation*}
\eqref{ds:Cauchy} leads to the following dynamical system associated to the multi-block optimization problem \eqref{intro:pb-mb}
\begin{equation*}
\begin{cases}
\ddot{x}_{1} \left( t \right) + \dfrac{\alpha}{t} \dot{x}_{1} \left( t \right) + \nabla f_{1} \left( x_{1} \left( t \right) \right) + A_{1}^{*} \left( \lambda \left( t \right) + \theta t \dot{\lambda} \left( t \right) \right) + \beta A_{1}^{*} \Bigl( \mysum_{i=1}^{m} A_{i}x_{i} \left( t \right) - b \Bigr)		& = 0 \\
\hfill \vdots \hfill \\
\ddot{x}_{m} \left( t \right) + \dfrac{\alpha}{t} \dot{x}_{m} \left( t \right) + \nabla f_{m} \left( x_{m} \left( t \right) \right) + A_{m}^{*} \left( \lambda \left( t \right) + \theta t \dot{\lambda} \left( t \right) \right) + \beta A_{m}^{*} \Bigl( \mysum_{i=1}^{m} A_{i}x_{i} \left( t \right) - b \Bigr)		& = 0 \\
\ddot{\lambda} \left( t \right) + \dfrac{\alpha}{t} \dot{\lambda} \left( t \right) - \Bigl( \mysum_{i=1}^{m} A_{i} \bigl( x_{i} \left( t \right) + \theta t \dot{x}_{i} \left( t \right) \bigr) - b \Bigr) 	& = 0\\
\Bigl(x_1(t_0), ..., x_m(t_0) , \lambda(t_0) \Bigr) 			= \Bigl( x_{10}, ..., x_{m0}, \lambda_{0} \Bigr) \textrm{ and } & \\
 \Bigl(\dot x_1(t_0), ..., \dot x_m(t_0) , \dot \lambda(t_0) \Bigr) 			= \Bigl(\dot x_{10}, ..., \dot x_{m0}, \dot \lambda_{0} \Bigr) &
\end{cases},
\end{equation*}
where $\alpha \geq 3$, $\beta \geq 0$, $\theta >0$ and $( x_{10}, ..., x_{m0}, \lambda_{0}), (\dot x_{10}, ..., \dot x_{m0}, \dot \lambda_{0}) \in \sX_1 \times ... \times \sX_m \times \sY$. By making use of the above construction, all results we will obtain in the paper for \eqref{intro:pb} can be transferred to the multi-block optimization problems \eqref{intro:pb-mb}.

\section{Fast convergence rates}\label{sec:Lya}

In this section we will derive fast convergence rates for the primal-dual gap, the feasibility measure, and the objective function value along the trajectories generated by the dynamical system \eqref{ds:PD-AVD}. Throughout this section we will make the following assumption on the parameters $\alpha$, $\beta$ and $\theta$.
\begin{mdframed}
	\begin{assume}
		\label{assume:para}
		Suppose that $\alpha, \beta$ and $\theta$ in \eqref{ds:PD-AVD} satisfy
		\begin{equation*}
		\alpha \geq 3, \quad  \beta \geq 0
		\quad \textrm{ and } \quad
		\dfrac{1}{2} \geq \theta \geq \dfrac{1}{\alpha - 1}.
		\end{equation*}
	\end{assume}
\end{mdframed}

\subsection{The energy function}

Let $\left( x , \lambda \right) \colon \left[ t_{0} , + \infty \right) \to \sX \times \sY$ be a solution of \eqref{ds:PD-AVD}. For 
$\left( z , \mu \right) \in \sX \times \sY$ fixed,  we define
\begin{equation*}
\G_{\beta} : \sX \times \sY \rightarrow \sR, \quad \G_{\beta} \Bigl( \bigl(x,\lambda \bigr) \Big\vert \left( z , \mu \right) \Bigr)  := \Lb \left(x , \mu \right) - \Lb \left( z , \lambda  \right) .
\end{equation*}
According to \eqref{intro:aug-Lag} and \eqref{intro:Fea:eq}, we have for every $\left( z , \mu \right) \in \Fea \times \sY$ and every $t \geq t_0$
\begin{equation*}
\G_{\beta} \Bigl( \bigl( x \left( t \right) , \lambda \left( t \right) \bigr) \Big\vert \left( z , \mu \right) \Bigr)	 
= f \left( x \left( t \right) \right) - f \left( z \right) + \left\langle \mu , Ax \left( t \right) - b \right\rangle + \dfrac{\beta}{2} \left\lVert Ax \left( t \right) - b \right\rVert ^{2} .
\end{equation*}
When $\left( z , \mu \right) := \left( x_{*} , \lambda_{*} \right) \in \sol$, it holds for every $t \geq t_0$
\begin{align}
\G_{\beta} \Bigl( \bigl( x \left( t \right) , \lambda \left( t \right) \bigr) \Big\vert \left( x_{*} , \lambda_{*} \right) \Bigr)	 
& = \Lb \left( x \left( t \right) , \lambda_{*} \right) - \Lb \left( x_{*} , \lambda \left( t \right) \right) \nonumber \\
& = \Lag \left( x \left( t \right) , \lambda_{*} \right) - \Lag \left( x_{*} , \lambda \left( t \right) \right) + \dfrac{\beta}{2} \left\lVert Ax \left( t \right) - b \right\rVert ^{2} \label{ener:PD-gap:Lag} \\
& = \Lag \left( x \left( t \right) , \lambda_{*} \right) - f(x_*) + \dfrac{\beta}{2} \left\lVert Ax \left( t \right) - b \right\rVert ^{2} \nonumber \\
& = f \left( x \left( t \right) \right) - f_{*} + \left\langle \lambda_{*} , Ax \left( t \right) - b \right\rangle + \dfrac{\beta}{2} \left\lVert Ax \left( t \right) - b \right\rVert ^{2} \geq 0 , \label{ener:PD-gap}
\end{align}
where $f_{*}$ denotes the optimal objective  value of \eqref{intro:pb}.

For $\left( z , \mu \right) \in \sX \times \sY$ fixed, we introduce the energy function ${\E_{z, \mu} \colon \left[ t_{0} , + \infty \right) \to \sR}$ defined as
\begin{equation}
\label{ener:E}
\E_{z, \mu} \left( t \right)
:= \theta^{2} t^{2} \G_{\beta} \Bigl( \bigl( x \left( t \right) , \lambda \left( t \right) \bigr) \Big\vert \left( z , \mu \right) \Bigr) + \dfrac{1}{2} \left\lVert v_{z , \mu} \left( t \right) \right\rVert ^{2} + \dfrac{\xi}{2} \left\lVert \bigl( x \left( t \right) , \lambda \left( t \right) \bigr) - \left( z , \mu \right) \right\rVert ^{2} ,
\end{equation}
where
\begin{align}
v_{z , \mu} \left( t \right)
& := \bigl( x \left( t \right) , \lambda \left( t \right) \bigr) - \left( z , \mu \right) + \theta t \left( \dot{x} \left( t \right) , \dot{\lambda} \left( t \right) \right) , \label{ener:v} \\
\xi & := \theta \alpha - \theta - 1 \geq 0 . \label{ener:xi}
\end{align}

Notice that due to \eqref{ener:PD-gap}, for $\left( x_{*} , \lambda_{*} \right) \in \sol$ we have
\begin{equation}
\label{ener:pos}
\E_{x_{*} , \lambda_{*}} \left( t \right) \geq 0 \qquad \forall t \geq t_{0} .
\end{equation}

\begin{lem}
	\label{lem:dec}
	Let $\left( x , \lambda \right) \colon \left[ t_{0} , + \infty \right) \to \sX \times \sY$ be a solution of \eqref{ds:PD-AVD} and $\left( z , \mu \right) \in \Fea \times \sY$. For every $t \geq t_{0}$ it holds
	\begin{align*}
	\dfrac{d}{dt} \E_{z, \mu} \left( t \right) \leq \left( 2 \theta - 1 \right) \theta t \G_{\beta} \Bigl( \bigl( x \left( t \right) , \lambda \left( t \right) \bigr) \Big\vert \left( z , \mu \right) \Bigr) - \dfrac{\beta \theta t}{2} \left\lVert Ax \left( t \right) - b \right\rVert ^{2} - \xi \theta t \left\lVert \left( \dot{x} \left( t \right) , \dot{\lambda} \left( t \right) \right) \right\rVert ^{2} .
	\end{align*}
\end{lem}
\begin{proof}
Let $t \geq t_{0}$ be fixed.
Since $z \in \Fea$, we have
\begin{align*}
\nabla \G_{\beta} \Bigl( \bigl( x \left( t \right) , \lambda \left( t \right) \bigr) \Big\vert \left( z , \mu \right) \Bigr)		
& = \begin{pmatrix}
\nabla_{x} \Lb \left( x \left( t \right) , \mu \right) , - \nabla_{\lambda} \Lb \left( z , \lambda \left( t \right) \right)
\end{pmatrix} = \begin{pmatrix}
\nabla_{x} \Lb \left( x \left( t \right) , \mu \right) , 0
\end{pmatrix} \nonumber \\
& = \begin{pmatrix}
\nabla f \left( x \left( t \right) \right) + A^* \mu + \beta A^* \left( Ax \left( t \right) - b \right) , 0
\end{pmatrix} .
\end{align*}

Differentiating $\E$ with respect to $t$ gives
\begin{align}
\dfrac{d}{dt} \E_{x_{*} , \lambda_{*}} \left( t \right) 
=  & \ 2 \theta^{2} t \G_{\beta} \Bigl( \bigl( x \left( t \right) , \lambda \left( t \right) \bigr) \Big\vert \left( z , \mu \right) \Bigr) + \theta^{2} t^{2} \left\langle \nabla \G_{\beta} \Bigl( \bigl( x \left( t \right) , \lambda \left( t \right) \bigr) \Big\vert \left( z , \mu \right) \Bigr) , \left( \dot{x} \left( t \right) , \dot{\lambda} \left( t \right) \right) \right\rangle \nonumber \\
& \ + \left\langle v_{z , \mu} \left( t \right) , \dot{v}_{z , \mu} \left( t \right) \right\rangle + \xi \left\langle \bigl( x \left( t \right) , \lambda \left( t \right) \bigr) - \left( z , \mu \right) , \left( \dot{x} \left( t \right) , \dot{\lambda} \left( t \right) \right) \right\rangle . \label{dec:dE}
\end{align}
The system \eqref{ds:PD-AVD} can be equivalently written as
\begin{align*}
\begin{pmatrix}
\ddot{x} \left( t \right) , \ddot{\lambda} \left( t \right)
\end{pmatrix} 
= & - \dfrac{\alpha}{t} \begin{pmatrix}
\dot{x} \left( t \right) , \dot{\lambda} \left( t \right)
\end{pmatrix} 
- \nabla \G_{\beta} \Bigl( \bigl( x \left( t \right) , \lambda \left( t \right) \bigr) \Big\vert \left( z , \mu \right) \Bigr) \nonumber \\
& - \begin{pmatrix}
A^* \left( \lambda \left( t \right) - \mu + \theta t \dot{\lambda} \left( t \right) \right) ,
- \Bigl( A \left( x \left( t \right) + \theta t \dot{x} \left( t \right) \right) - b \Bigr)
\end{pmatrix} ,
\end{align*}
which leads to
\begin{align*}
\dot{v}_{z , \mu} \left( t \right)
= & \left( 1 + \theta \right) \left( \dot{x} \left( t \right) , \dot{\lambda} \left( t \right) \right) + \theta t \left( \ddot{x} \left( t \right) , \ddot{\lambda} \left( t \right) \right) \nonumber \\
= & - \xi \left( \dot{x} \left( t \right) , \dot{\lambda} \left( t \right) \right) - \theta t \nabla \G_{\beta} \Bigl( \bigl( x \left( t \right) , \lambda \left( t \right) \bigr) \Big\vert \left( z , \mu \right) \Bigr) \nonumber \\
& - \theta t \begin{pmatrix}
A^* \left( \lambda \left( t \right) - \mu + \theta t \dot{\lambda} \left( t \right) \right) ,
- \Bigl( A \left( x \left( t \right) + \theta t \dot{x} \left( t \right) \right) - b \Bigr)
\end{pmatrix}.
\end{align*}
We get from the distributive property of inner product
\begin{align*}
& \left\langle v_{z , \mu} \left( t \right) , \dot{v}_{z , \mu} \left( t \right) \right\rangle \nonumber \\
= \ 	& - \xi \left\langle \bigl( x \left( t \right) , \lambda \left( t \right) \bigr) - \left( z , \mu \right) , \left( \dot{x} \left( t \right) , \dot{\lambda} \left( t \right) \right) \right\rangle 
- \xi \theta t \left\lVert \left( \dot{x} \left( t \right) , \dot{\lambda} \left( t \right) \right) \right\rVert ^{2} \nonumber \\
& - \theta t \left\langle \nabla \G_{\beta} \Bigl( \bigl( x \left( t \right) , \lambda \left( t \right) \bigr) \Big\vert \left( z , \mu \right) \Bigr) , \bigl( x \left( t \right) , \lambda \left( t \right) \bigr) - \left( z , \mu \right) \right\rangle \nonumber \\
& - \theta^{2} t^{2} \left\langle \nabla \G_{\beta} \Bigl( \bigl( x \left( t \right) , \lambda \left( t \right) \bigr) \Big\vert \left( z , \mu \right) \Bigr) , \left( \dot{x} \left( t \right) , \dot{\lambda} \left( t \right) \right) \right\rangle 
- \theta t \left\langle \lambda \left( t \right) - \mu + \theta t \dot{\lambda} \left( t \right) , Ax \left( t \right) - Az \right\rangle \nonumber \\
& - \theta^{2} t^{2} \left\langle \lambda \left( t \right) - \mu + \theta t \dot{\lambda} \left( t \right) , A \dot{x} \left( t \right) \right\rangle 
+ \theta t \left\langle A \left( x \left( t \right) + \theta t \dot{x} \left( t \right) \right) - b , \lambda \left( t \right) - \mu \right\rangle \nonumber \\
& + \theta^{2} t^{2} \left\langle A \left( x \left( t \right) + \theta t \dot{x} \left( t \right) \right) - b , \dot{\lambda} \left( t \right) \right\rangle .
\end{align*}
Since $z \in \Fea$, the last four terms in the above identity vanish. Indeed,
\begin{align*}
& - \left\langle \lambda \left( t \right) - \mu + \theta t \dot{\lambda} \left( t \right) , Ax \left( t \right) - Az \right\rangle 
- \theta t \left\langle \lambda \left( t \right) - \mu + \theta t \dot{\lambda} \left( t \right) , A \dot{x} \left( t \right) \right\rangle \nonumber \\
&  + \left\langle A \left( x \left( t \right) + \theta t \dot{x} \left( t \right) \right) - b , \lambda \left( t \right) - \mu \right\rangle 
+ \theta t \left\langle A \left( x \left( t \right) + \theta t \dot{x} \left( t \right) \right) - b , \dot{\lambda} \left( t \right) \right\rangle \nonumber \\
= \ 	& - \left\langle \lambda \left( t \right) - \mu + \theta t \dot{\lambda} \left( t \right) , Ax \left( t \right) - b \right\rangle 
- \theta t \left\langle \lambda \left( t \right) - \mu + \theta t \dot{\lambda} \left( t \right) , A \dot{x} \left( t \right) \right\rangle \nonumber \\
& + \!\left\langle Ax \left( t \right) - b , \lambda \left( t \right) - \mu \right\rangle 
+ \theta t \left\langle A \dot{x} \left( t \right) , \lambda \left( t \right) - \mu \right\rangle \nonumber + \theta t \left\langle Ax \left( t \right) - b , \dot{\lambda} \left( t \right) \right\rangle 
+ \theta^{2} t^{2} \left\langle A \dot{x} \left( t \right) , \dot{\lambda} \left( t \right) \right\rangle \\
=  \ & \ 0 .
\end{align*}
Therefore, \eqref{dec:dE} becomes
\begin{align}
\dfrac{d}{dt} \E_{x_{*} , \lambda_{*}} \left( t \right)
= & \ 2 \theta^{2} t \G_{\beta} \Bigl( \bigl( x \left( t \right) , \lambda \left( t \right) \bigr) \Big\vert \left( z , \mu \right) \Bigr) - \xi \theta t \left\lVert \left( \dot{x} \left( t \right) , \dot{\lambda} \left( t \right) \right) \right\rVert ^{2} \nonumber \\
& - \theta t \left\langle \nabla \G_{\beta} \left( \bigl( x \left( t \right) , \lambda \left( t \right) \bigr) \vert \left( z , \mu \right) \right) , \bigl( x \left( t \right) , \lambda \left( t \right) \bigr) - \left( z , \mu \right) \right\rangle . \label{dec:pre}
\end{align}
Furthermore, the convexity of $f$ and the fact that $z \in \Fea$ guarantee
\begin{align}
& - \left\langle \nabla \G_{\beta} \Bigl( \bigl( x \left( t \right) , \lambda \left( t \right) \bigr) \Big\vert \left( z , \mu \right) \Bigr) , \bigl( x \left( t \right) , \lambda \left( t \right) \bigr) - \left( z , \mu \right) \right\rangle \nonumber \\
= \ 	& \left\langle \nabla f \left( x \left( t \right) \right) , z - x \left( t \right) \right\rangle 
+ \left\langle A^* \mu , z - x \left( t \right) \right\rangle
+ \beta \left\langle A^* \left( Ax \left( t \right) - b \right) , z - x \left( t \right) \right\rangle \nonumber \\
\leq \ 	& - \left( f \left( x \left( t \right) \right) - f \left( z \right) \right) - \left\langle \mu , Ax \left( t \right) - b \right\rangle - \beta \left\lVert Ax \left( t \right) - b \right\rVert ^{2} \label{dec:conv} \\
= \ 	& - \G_{\beta} \Bigl( \bigl( x \left( t \right) , \lambda \left( t \right) \bigr) \Big\vert \left( z , \mu \right) \Bigr) - \dfrac{\beta}{2} \left\lVert Ax \left( t \right) - b \right\rVert ^{2} \nonumber.
\end{align}
Combining this inequality with \eqref{dec:pre} yields the desired statement.
\end{proof}
An important  consequence of Lemma \ref{lem:dec} is the following theorem.

\begin{thm}
	\label{thm:bnd}
	Let $\left( x , \lambda \right) \colon \left[ t_{0} , + \infty \right) \to \sX \times \sY$ be a solution of \eqref{ds:PD-AVD} and $\left( x_{*} , \lambda_{*} \right) \in \sol$.	The following statements are true:
	\begin{enumerate}
		\item 		
		it holds
%
		\begin{align}		
		\beta \int_{t_{0}}^{+ \infty} t \left\lVert Ax \left( t \right) - b \right\rVert ^{2} dt 
		& \leq {\dfrac{2 \E_{x_{*} , \lambda_{*}} \left( t_{0} \right)}{\theta}} < + \infty , \label{bnd:int:fea} \\
		\left( 1 - 2 \theta \right) \int_{t_{0}}^{+ \infty} t \Bigl( \Lag \left( x \left( t \right) , \lambda_{*} \right) - \Lag \left( x_{*} , \lambda \left( t \right) \right) \Bigr) dt 
		& \leq \dfrac{\E_{x_{*} , \lambda_{*}} \left( t_{0} \right)}{\theta} < + \infty , \label{bnd:int:Lag} \\
		\xi \int_{t_{0}}^{+ \infty} t \left\lVert \left( \dot{x} \left( t \right) , \dot{\lambda} \left( t \right) \right) \right\rVert ^{2} dt
		& \leq \dfrac{\E_{x_{*} , \lambda_{*}} \left( t_{0} \right)}{\theta} < + \infty \label{bnd:int:vel} ;
		\end{align}
		
		\item 
		if, in addition $\alpha > 3$ and $\frac{1}{2} \geq \theta > \frac{1}{\alpha - 1}$, then the trajectory $\bigl( x \left( t \right) , \lambda \left( t \right) \bigr)_{t \geq t_0}$ is bounded and the convergence rate of its velocity is
		\begin{equation*}
		\left\lVert \left( \dot{x} \left( t \right) , \dot{\lambda} \left( t \right) \right)  \right\rVert = \bO \left( \dfrac{1}{t} \right) \quad \mbox{as} \quad t \rightarrow +\infty.
		\end{equation*}
	\end{enumerate} 
\end{thm}
\begin{proof}
	\begin{enumerate}
		\item 
		Assumption \ref{assume:para} implies that $2 \theta - 1 \leq 0$ and $\xi \geq 0$ (see \eqref{ener:xi}).
		Moreover, $\left( x_{*} , \lambda_{*} \right) \in \sol$ yields $x_{*} \in \Fea$. 
		Therefore, we can apply Lemma \ref{lem:dec} to obtain for every $t \geq t_{0}$ 
		\begin{align}		
		\dfrac{d}{dt} \E_{x_{*} , \lambda_{*}} \left( t \right) 
		& \leq \left( 2 \theta - 1 \right) \theta t \Bigl( \Lb \left( x \left( t \right) , \lambda_{*} \right) - \Lb \left( x_{*} , \lambda \left( t \right) \right) \Bigr) \nonumber \\
		& \qquad - \dfrac{\beta \theta t}{2} \left\lVert Ax \left( t \right) - b \right\rVert ^{2} - \xi \theta t \left\lVert \left( \dot{x} \left( t \right) , \dot{\lambda} \left( t \right) \right) \right\rVert ^{2} \nonumber \\
		& \leq \left( 2 \theta - 1 \right) \theta t \Bigl( \Lag \left( x \left( t \right) , \lambda_{*} \right) - \Lag \left( x_{*} , \lambda \left( t \right) \right) \Bigr) \nonumber \\
		& \qquad - \dfrac{\beta \theta t}{2} \left\lVert Ax \left( t \right) - b \right\rVert ^{2} 
		- \xi \theta t \left\lVert \left( \dot{x} \left( t \right) , \dot{\lambda} \left( t \right) \right) \right\rVert ^{2} \nonumber\\
& \leq 0 . \label{bnd:dec}
		\end{align}
		This means that $\E_{x_{*} , \lambda_{*}}$ is nonincreasing on $\left[ t_{0} , + \infty \right)$, thus, for every $t \geq t_0$ it holds
		\begin{align}		
		& \theta^{2} t^{2} \Bigl( \Lb \left( x \left( t \right) , \lambda_{*} \right) - \Lb \left( x_{*} , \lambda \left( t \right) \right) \Bigr)  + \dfrac{1}{2} \left\lVert v_{x_{*} , \lambda_{*}} \left( t \right) \right\rVert ^{2} + \dfrac{\xi}{2} \left\lVert \bigl( x \left( t \right) , \lambda \left( t \right) \bigr) - \left( x_{*} , \lambda_{*} \right) \right\rVert ^{2} \nonumber \\
		\leq & \ \E_{x_{*} , \lambda_{*}}  \left( t_{0} \right). \label{bnd:fun}
		\end{align}
		
		For every $t \geq t_{0}$, by integrating \eqref{bnd:dec} from $t_{0}$ to $t$, we obtain
		\begin{align*}
		& \left( 1 - 2 \theta \right) \theta \int_{t_{0}}^{t} s \Bigl( \Lag \left( x \left( s \right) , \lambda_{*} \right) - \Lag \left( x_{*} , \lambda \left( s \right) \right) \Bigr) ds \nonumber \\
		& + \dfrac{\beta \theta}{2} \int_{t_{0}}^{t} s \left\lVert Ax \left( s \right) - b \right\rVert ^{2} ds
		+ \xi \theta \int_{t_{0}}^{t} s \left\lVert \left( \dot{x} \left( s \right) , \dot{\lambda} \left( s \right) \right) \right\rVert ^{2} ds \nonumber \\
		\leq & \ \E_{x_{*} , \lambda_{*}} \left( t_{0} \right) - \E_{x_{*} , \lambda_{*}} \left( t \right) \leq \E_{x_{*} , \lambda_{*}} \left( t_{0} \right) ,
		\end{align*}
where the last inequality follows from \eqref{ener:pos}. Since all quantities inside the integrals are nonnegative, we obtain \eqref{bnd:int:fea} - \eqref{bnd:int:vel} by passing $t \to + \infty$.
		
		\item 
		Assuming that $\alpha > 3$ and $\frac{1}{2} \geq \theta > \frac{1}{\alpha - 1}$, one can immediately see that $\xi > 0$. From \eqref{bnd:fun} we obtain for all $t \geq t_0$
		\begin{equation}
		\label{bnd:tra-sol}
		\left\lVert \bigl( x \left( t \right) , \lambda \left( t \right) \bigr) - \left( x_{*} , \lambda_{*} \right) \right\rVert ^{2}
		\leq \dfrac{2 \E_{x_{*} , \lambda_{*}} \left( t_{0} \right)}{\xi} \quad \forall t \geq t_0,
		\end{equation}
		which implies the boundedness of the trajectory. 
		On the other hand, the same inequality gives  for all $t \geq t_0$
		\begin{equation}
		\label{bnd:v}
		\left\lVert v_{x_{*} , \lambda_{*}} \left( t \right) \right\rVert = \left\lVert \bigl( x \left( t \right) , \lambda \left( t \right) \bigr) - \left( x_{*} , \lambda_{*} \right) + \theta t \left( \dot{x} \left( t \right) , \dot{\lambda} \left( t \right) \right) \right\rVert \leq \sqrt{2 \E_{x_{*} , \lambda_{*}} \left( t_{0} \right)} .
		\end{equation}
		Using the triangle inequality and \eqref{bnd:tra-sol} we obtain  for all $t \geq t_0$
		\begin{align}
		t \left\lVert \dot{x} \left( t \right) , \dot{\lambda} \left( t \right) \right\rVert
		& \leq \dfrac{1}{\theta} \left( \left\lVert \bigl( x \left( t \right) , \lambda \left( t \right) \bigr) - \left( x_{*} , \lambda_{*} \right) \right\rVert + \left\lVert v_{x_{*} , \lambda_{*}} \left( t \right) \right\rVert \right) \nonumber \\
		& \leq \dfrac{1}{\theta} \left( \sqrt{\dfrac{2 \E_{x_{*} , \lambda_{*}} \left( t_{0} \right)}{\xi}} + \sqrt{2 \E_{x_{*} , \lambda_{*}} \left( t_{0} \right)} \right) 
		= \dfrac{1}{\theta} \left( \dfrac{1}{\sqrt{\xi}} + 1 \right) \sqrt{2 \E_{x_{*} , \lambda_{*}} \left( t_{0} \right)} , \label{bnd:der}
		\end{align}
		which gives the desired convergence rate.
		\qedhere
	\end{enumerate}
\end{proof}

\subsection{Fast convergence rates for the primal-dual gap, the feasibility measure and the objective function value}

The following result quantifies the values of the energy function when defined with respect to a primal-dual element which slightly deviates from an element in $\sol$.

\begin{lem}
	\label{lem:rate}
	Let $\left( x , \lambda \right) \colon \left[ t_{0} , + \infty \right) \to \sX \times \sY$ be a solution of \eqref{ds:PD-AVD} and $\left( x_{*} , \lambda_{*} \right) \in \sol$. The following statements are true:
	\begin{enumerate}
		\item 
		the following quantity is finite
		\begin{equation}
		\label{rate:sup}
		\Csup 	:= \sup_{\mu \in \sB \left( \lambda_{*} ; 1 \right)} \E_{x_{*} , \mu} \left( t_{0} \right) < + \infty ;
		\end{equation}
		
		\item 
		for every $\mu \in \sB \left( \lambda_{*} ; 1 \right)$ and every $t \geq t_0$ it holds
		\begin{equation}
		\label{rate:est}
		\E_{x_{*} , \mu} \left( t \right) \leq 2 \E_{x_{*} , \lambda_{*}} \left( t_{0} \right) + \theta \left( \alpha - 1 \right) + \theta ^{2} t^{2} \left\langle \mu - \lambda_{*} , Ax \left( t \right) - b \right\rangle .
		\end{equation}
	\end{enumerate}
\end{lem}
\begin{proof}
	\begin{enumerate}		
		\item 		
		Let $\mu \in \sB \left( \lambda_{*} ; 1 \right)$. For every $t \geq t_{0}$ we have
		\begin{align}
		\E_{x_{*} , \mu} \left( t \right)
		= & \ \theta^{2} t^{2} \Bigl( \Lb \left( x \left( t \right) , \mu \right) - \Lb \left( x_{*} , \lambda \left( t \right) \right) \Bigr) + \dfrac{1}{2} \left\lVert v_{x_{*}, \mu} \left( t \right) \right\rVert ^{2} + \dfrac{\xi}{2} \left\lVert \bigl( x \left( t \right) , \lambda \left( t \right) \bigr) - \left( x_{*} , \mu \right) \right\rVert ^{2} \nonumber \\
		= & \ \theta^{2} t^{2} \Bigl( f \left( x \left( t \right) \right) - f \left( x_{*} \right) + \left\langle \mu , A x \left( t \right) - b \right\rangle + \dfrac{\beta}{2} \left\lVert A x \left( t \right) - b \right\rVert ^{2} \Bigr) \nonumber \\
		& + \dfrac{1}{2} \left\lVert \bigl( x \left( t \right) , \lambda \left( t \right) \bigr) - \left( x_{*} , \mu \right) + \theta t \left( \dot{x} \left( t \right) , \dot{\lambda} \left( t \right) \right) \right\rVert ^{2} + \dfrac{\xi}{2} \left\lVert \bigl( x \left( t \right) , \lambda \left( t \right) \bigr) - \left( x_{*} , \mu \right) \right\rVert ^{2}. \label{rate:Ener} 
		\end{align}
By the Cauchy-Schwarz inequality we get
		\begin{align}
		& f \left( x \left( t_{0} \right) \right) - f \left( x_{*} \right) + \left\langle \mu , A x \left( t_{0} \right) - b \right\rangle + \dfrac{\beta}{2} \left\lVert A x \left( t_{0} \right) - b \right\rVert ^{2} \nonumber \\
		\leq \ 	& f \left( x \left( t_{0} \right) \right) - f \left( x_{*} \right) + \left\lVert \mu \right\rVert \cdot \left\lVert A x \left( t_{0} \right) - b \right\rVert + \dfrac{\beta}{2} \left\lVert A x \left( t_{0} \right) - b \right\rVert ^{2} \nonumber \\
		\leq \ 	& \CLag := \left\lvert f \left( x \left( t_{0} \right) \right) - f \left( x_{*} \right) \right\rvert + \left( 1 + \left\lVert \lambda_{*} \right\rVert \right) \cdot \left\lVert A x \left( t_{0} \right) - b \right\rVert + \dfrac{\beta}{2} \left\lVert A x \left( t_{0} \right) - b \right\rVert ^{2} . \label{rate:sup:Lag}
		\end{align}
		We also have
		\begin{align}
		& \dfrac{1}{2} \left\lVert \left( x \left( t_{0} \right) , \lambda \left( t_{0} \right) \right) - \left( x_{*} , \mu \right) + \theta t_0 \left( \dot{x} \left( t_{0} \right) , \dot{\lambda} \left( t_{0} \right) \right) \right\rVert ^{2} \nonumber \\
		\leq \ 	& \left\lVert \left( x \left( t_{0} \right) , \lambda \left( t_{0} \right) \right) - \left( x_{*} , \lambda_{*} \right) + \theta t_0 \left( \dot{x} \left( t_{0} \right) , \dot{\lambda} \left( t_{0} \right) \right) \right\rVert ^{2} + \left\lVert \mu - \lambda_{*} \right\rVert ^{2} \nonumber \\
		\leq \ 	& \Cv := \left\lVert \left( x \left( t_{0} \right) , \lambda \left( t_{0} \right) \right) - \left( x_{*} , \lambda_{*} \right) + \theta t_0 \left( \dot{x} \left( t_{0} \right) , \dot{\lambda} \left( t_{0} \right) \right) \right\rVert ^{2} + 1\label{rate:sup:v}
		\end{align}
and
		\begin{align}
		\dfrac{1}{2} \left\lVert \left( x \left( t_{0} \right) , \lambda \left( t_{0} \right) \right) - \left( x_{*} , \mu \right) \right\rVert ^{2}
		& \leq \left\lVert \left( x \left( t_{0} \right) , \lambda \left( t_{0} \right) \right) - \left( x_{*} , \lambda_{*} \right) \right\rVert ^{2} + \left\lVert \mu - \lambda_{*} \right\rVert ^{2} \nonumber \\
		& \leq \Cite := \left\lVert \left( x \left( t_{0} \right) , \lambda \left( t_{0} \right) \right) - \left( x_{*} , \lambda_{*} \right) \right\rVert ^{2} + 1 . \label{rate:sup:ite}
		\end{align}
		Combining \eqref{rate:sup:Lag} - \eqref{rate:sup:ite}, it yields
		\begin{equation*}
		\Csup = \sup_{\mu \in \sB \left( \lambda_{*} ; 1 \right)} \E_{x_{*} , \mu} \left( t_{0} \right) 
		\leq \theta^{2} t_{0}^{2} \CLag + \Cv + \xi \Cite < + \infty ,
		\end{equation*}
		which proves \eqref{rate:sup}.
		
		\item 
		Let $t \geq t_{0}$.
		By recalling \eqref{ener:v} and \eqref{bnd:fun} we easily see that
		\begin{align}
		& \dfrac{1}{2} \left\lVert v_{x_{*} , \mu} \left( t \right) \right\rVert ^{2} + \dfrac{\xi}{2} \left\lVert \bigl( x \left( t \right) , \lambda \left( t \right) \bigr) - \left( x_{*} , \mu \right) \right\rVert ^{2} \nonumber \\
		\leq \ 	& \left\lVert v_{x_{*} , \lambda_{*}} \left( t \right) \right\rVert ^{2} + \xi \left\lVert \bigl( x \left( t \right) , \lambda \left( t \right) \bigr) - \left( x_{*} , \lambda_{*} \right) \right\rVert ^{2} + \left( 1 + \xi \right) \left\lVert \mu - \lambda_{*} \right\rVert ^{2} \nonumber \\
		\leq \ 	& \dfrac{1}{2} \left\lVert v_{x_{*} , \lambda_{*}} \left( t \right) \right\rVert ^{2} + \dfrac{\xi}{2} \left\lVert \bigl( x \left( t \right) , \lambda \left( t \right) \bigr) - \left( x_{*} , \lambda_{*} \right) \right\rVert ^{2} 
		+ \E_{x_{*} , \lambda_{*}} \left( t_{0} \right) + 1 + \xi . \label{rate:const}
		\end{align}
		Furthermore, by the definition of $\G_{\beta}$ and relation \eqref{ener:PD-gap:Lag} we have that
		\begin{align}
		\G_{\beta} \Bigl( \bigl( x \left( t \right) , \lambda \left( t \right) \bigr) \Big\vert \left( x_{*} , \mu \right) \Bigr)
		= & \ f \left( x \left( t \right) \right) - f(x_*) + \left\langle \lambda_{*} , Ax \left( t \right) - b \right\rangle + \dfrac{\beta}{2} \left\lVert Ax \left( t \right) - b \right\rVert ^{2} \nonumber \\
		& + \left\langle \mu - \lambda_{*} , Ax \left( t \right) - b \right\rangle \nonumber \\
		= & \ \G_{\beta} \Bigl( \bigl( x \left( t \right) , \lambda \left( t \right) \bigr) \Big\vert \left( x_{*} , \lambda_{*} \right) \Bigr)	
		+ \left\langle \mu - \lambda_{*} , Ax \left( t \right) - b \right\rangle \label{rate:gap-beta:eq} \\
		\geq & \left\langle \mu - \lambda_{*} , Ax \left( t \right) - b \right\rangle . \label{rate:gap-beta}
		\end{align}
Relations \eqref{rate:const} and \eqref{rate:gap-beta:eq} lead to
		\begin{align*}
		\E_{x_{*} , \mu} \left( t \right)
		= & \ \theta ^{2} t^{2} \G_{\beta} \Bigl( \bigl( x \left( t \right) , \lambda \left( t \right) \bigr) \Big\vert \left( x_{*} , \lambda_{*} \right) \Bigr) 
		+ \theta ^{2} t^{2} \left\langle \mu - \lambda_{*} , Ax \left( t \right) - b \right\rangle \nonumber \\
		& + \dfrac{1}{2} \left\lVert v_{x_{*} , \mu} \left( t \right) \right\rVert ^{2} + \dfrac{\xi}{2} \left\lVert \bigl( x \left( t \right) , \lambda \left( t \right) \bigr) - \left( x_{*} , \mu \right) \right\rVert ^{2} \nonumber \\
		 \leq & \ \E_{x_{*} , \lambda_{*}} \left( t \right) + \theta ^{2} t^{2} \left\langle \mu - \lambda_{*} , Ax \left( t \right) - b \right\rangle + \E_{x_{*} , \lambda_{*}} \left( t_{0} \right) + 1 + \xi \nonumber \\
		\leq & \ 2 \E_{x_{*} , \lambda_{*}} \left( t_{0} \right) + \theta \left( \alpha - 1 \right) + \theta ^{2} t^{2} \left\langle \mu - \lambda_{*} , Ax \left( t \right) - b \right\rangle ,
		\end{align*}
		where the last inequality is due to \eqref{ener:xi} and \eqref{bnd:fun}. This is nothing else than \eqref{rate:est}.
		\qedhere
	\end{enumerate}
\end{proof}

We can now formulate and prove the main convergence rate results of the paper
\begin{thm}
	\label{thm:rate}
	Let $\left( x , \lambda \right) \colon \left[ t_{0} , + \infty \right) \to \sX \times \sY$ be a solution of \eqref{ds:PD-AVD} and $\left( x_{*} , \lambda_{*} \right) \in \sol$. The following statements are true:
	\begin{enumerate}				
		\item
		for every $t \geq t_0$ it holds
		\begin{equation}
		\label{rate:gap}
		0 \leq \Lag \left( x \left( t \right) , \lambda_{*} \right) - \Lag \left( x_{*} , \lambda \left( t \right) \right) + \left\lVert A x \left( t \right) - b \right\rVert \leq \dfrac{\Cbnd}{\theta^{2} t^{2}} ,
		\end{equation}
		where
		\begin{equation}
		\label{rate:C4}
		\Cbnd := \Csup + 2 \E_{x_{*} , \lambda_{*}} \left( t_{0} \right) + \theta \left( \alpha - 1 \right) > 0 ;
		\end{equation}
		
		\item 
		for every $t \geq t_0$ it holds
		\begin{equation}
		\label{rate:fun}
		- \dfrac{\left\lVert \lambda_{*} \right\rVert \Cbnd}{\theta^{2} t^{2}} \leq f \left( x \left( t \right) \right) - f_{*} \leq \dfrac{\left( 1 + \left\lVert \lambda_{*} \right\rVert \right) \Cbnd}{\theta^{2} t^{2}} .
		\end{equation}
	\end{enumerate}
\end{thm}
\begin{proof}
	\begin{enumerate}					
		\item 			
		We fix $s \geq t_{0}$ and define
		\begin{equation}
		\label{rate:mu-s}
		\mu \left( s \right) := \begin{cases}
		\lambda_{*} + \dfrac{A x \left( s \right) - b}{\left\lVert A x \left( s \right) - b \right\rVert}, & \textrm{ if } A x \left( s \right) - b \neq 0 , \\
		\lambda_{*}, & \textrm{ if } A x \left( s \right) - b = 0 .
		\end{cases}
		\end{equation}
		It is clear that $\mu \left( s \right) \in \sB \left( \lambda_{*} ; 1 \right)$. For brevity, we set 
$$\sigma := \dfrac{1 - 2 \theta}{\theta} \geq 0.$$
Since $\left( x_{*} , \lambda_{*} \right) \in \sol$, we have $\left( x_{*} , \mu \left( s \right) \right) \in \Fea \times \sB \left( \lambda_{*} ; 1 \right)$. 		
		Lemma \ref{lem:dec} combined with the relation \eqref{rate:gap-beta} ensure that for every $t \geq t_{0}$ it holds
		\begin{align}
		\dfrac{d}{dt} \E_{x_{*} , \mu \left( s \right)} \left( t \right) 
		& \leq - \sigma \theta^{2} t \G_{\beta} \Bigl( \bigl( x \left( t \right) , \lambda \left( t \right) \bigr) \Big\vert \left( x_{*} , \mu \left( s \right) \right) \Bigr) \nonumber \\
& \leq - \sigma \theta^{2} t \left\langle \mu \left( s \right) - \lambda_{*} , Ax \left( t \right) - b \right\rangle . \label{rate:dec}
		\end{align}
We will prove that for every $t \geq t_{0}$ it holds
		\begin{align}
		\theta^{2} t^{2} \Bigl( f \left( x \left( t \right) \right) - f \left( x_{*} \right) + \left\langle \mu \left( s \right) , A x \left( t \right) - b \right\rangle \Bigr) 
		& \leq \E_{x_{*} , \mu \left( s \right)} \left( t \right) \nonumber \\
		& \leq \Cbnd = \Csup + 2 \E_{x_{*} , \lambda_{*}} \left( t_{0} \right) + \theta \left( \alpha - 1 \right) . \label{rate:t}
		\end{align}
		The first inequality follows from the definition of $\E_{x_{*} , \mu \left( s \right)}$. To show the later one, we multiply both sides  of \eqref{rate:dec} by $t^{\sigma} > 0$ and use integration by parts, for $\sigma >0$, or just integrate \eqref{rate:t}, for $\sigma=0$, to deduce that for every $t \geq t_0$
		\begin{align}
		t^{\sigma} \E_{x_{*} , \mu \left( s \right)} \left( t \right) - t_{0}^{\sigma} \E_{x_{*} , \mu \left( s \right)} \left( t_{0} \right) 
		& - \sigma \int_{t_{0}}^{t} \tau^{\sigma - 1} \E_{x_{*} , \mu \left( s \right)} \left( \tau \right) d \tau \nonumber \\
		& \leq - \sigma \theta^{2} \int_{t_{0}}^{t} \tau^{\sigma+1} \left\langle \mu \left( s \right) - \lambda_{*} , Ax \left( \tau \right) - b \right\rangle d \tau . \label{rate:yt}
		\end{align}
By using \eqref{rate:sup} and \eqref{rate:est} we further obtain for every $t \geq t_0$
		\begin{align*}
		t^{\sigma} \E_{x_{*} , \mu \left( s \right)} \left( t \right)
		 \leq & \ t_{0}^{\sigma} \E_{x_{*} , \mu \left( s \right)} \left( t_{0} \right)
		+ \sigma \int_{t_{0}}^{t} \tau^{\sigma - 1} \E_{x_{*} , \mu \left( s \right)} \left( \tau \right) d \tau \nonumber \\
		&  - \sigma \theta^{2} \int_{t_{0}}^{t} \tau^{\sigma+1} \left\langle \mu \left( s \right) - \lambda_{*} , Ax \left( \tau \right) - b \right\rangle d \tau \nonumber \\
		\leq & \ t_{0}^{\sigma} \Csup 
		+ \sigma \left( 2 \E_{x_{*} , \lambda_{*}} \left( t_{0} \right) + \theta \left( \alpha - 1 \right) \right) \int_{t_{0}}^{t} \tau^{\sigma - 1} d \tau \nonumber \\
		= & \ t_{0}^{\sigma} \Csup + \left( 2 \E_{x_{*} , \lambda_{*}} \left( t_{0} \right) + \theta \left( \alpha - 1 \right) \right) \bigl( t^{\sigma} - t_{0}^{\sigma} \bigr) \nonumber \\
		 \leq & \ t^{\sigma} \bigl( \Csup + 2 \E_{x_{*} , \lambda_{*}} \left( t_{0} \right) + \theta \left( \alpha - 1 \right) \bigr) ,
		\end{align*}
		which is equivalent to \eqref{rate:t}. 
		
Now, since \eqref{rate:t} is true for every $t \geq t_{0}$, it is fulfilled also for $t := s \geq t_{0}$, which means that
		\begin{equation*}
		\theta^{2} s^{2} \Bigl( f \left( x \left( s \right) \right) - f \left( x_{*} \right) + \left\langle \mu \left( s \right) , A x \left( s \right) - b \right\rangle \Bigr) \leq \Cbnd .
		\end{equation*}
By the definition of $\mu \left( s \right)$ in \eqref{rate:mu-s}, if $A x \left( s \right) - b \neq 0$, we have
		\begin{align*}
		& f \left( x \left( s \right) \right) - f \left( x_{*} \right) + \left\langle \mu \left( s \right) , A x \left( s \right) - b \right\rangle \nonumber \\
		= \ 	& f \left( x \left( s \right) \right) - f \left( x_{*} \right) + \left\langle \lambda_{*} , A x \left( s \right) - b \right\rangle + \left\lVert A x \left( s \right) - b \right\rVert \nonumber \\
		= \ 	& \Lag \left( x \left( s \right) , \lambda_{*} \right) - \Lag \left( x_{*} , \lambda \left( s \right) \right) + \left\lVert A x \left( s \right) - b \right\rVert,
		\end{align*}
while, if $A x \left( s \right) - b = 0$, we can also write
		\begin{align*}
		& f \left( x \left( s \right) \right) - f \left( x_{*} \right) + \left\langle \mu \left( s \right) , A x \left( s \right) - b \right\rangle \nonumber =  f \left( x \left( s \right) \right) - f \left( x_{*} \right) + \left\langle \lambda_{*} , A x \left( s \right) - b \right\rangle \nonumber \\
		= \ 	& \Lag \left( x \left( s \right) , \lambda_{*} \right) - \Lag \left( x_{*} , \lambda \left( s \right) \right) = \Lag \left( x \left( s \right) , \lambda_{*} \right) - \Lag \left( x_{*} , \lambda \left( s \right) \right) + \left\lVert A x \left( s \right) - b \right\rVert .
		\end{align*}
		For both scenarios, the estimate \eqref{rate:t} becomes
		\begin{equation*}
		\theta^{2} s^{2} \Bigl( \Lag \left( x \left( s \right) , \lambda_{*} \right) - \Lag \left( x_{*} , \lambda \left( s \right) \right) + \left\lVert A x \left( s \right) - b \right\rVert \Bigr) 
		\leq \Cbnd.
		\end{equation*}
Since $s \geq t_0$ has been arbitrarily chosen, this gives proves \eqref{rate:gap}.
		
		\item 
		Since $\Lag \left( x \left( t \right) , \lambda_{*} \right) - \Lag \left( x_{*} , \lambda \left( t \right) \right) \geq 0$, a direct consequent of \eqref{rate:gap} is that for every $t \geq t_0$
		\begin{equation}
		\label{rate:fea}
		\left\lVert A x \left( t \right) - b \right\rVert \leq \dfrac{\Cbnd}{\theta^{2} t^{2}} .
		\end{equation}		
		From \eqref{rate:gap} and the Cauchy-Schwarz inequality we can also deduce for every $t \geq t_0$ that
		\begin{align}
		f \left( x \left( t \right) \right) - f \left( x_{*} \right) 
		\leq \dfrac{\Cbnd}{\theta^{2} t^{2}} - \left\langle \lambda_{*} , Ax \left( t \right) - b \right\rangle
		& \leq \dfrac{\Cbnd}{\theta^{2} t^{2}} + \left\lVert \lambda_{*} \right\rVert \left\lVert Ax \left( t \right) - b \right\rVert \nonumber \\
		& \leq \dfrac{\left( 1 + \left\lVert \lambda_{*} \right\rVert \right) \Cbnd}{\theta^{2} t^{2}} . \label{rate:upper}
		\end{align}		
		On the other hand, the convexity of $f$ together with the fact that $\left( x_{*} , \lambda_{*} \right) \in \sol$ guarantee for every $t \geq t_0$
		\begin{align}
		f \left( x \left( t \right) \right) - f \left( x_{*} \right) 
		& \geq \left\langle \nabla f \left( x_{*} \right) , x \left( t \right) - x_{*} \right\rangle = - \left\langle A^* \lambda_{*} , x \left( t \right) - x_{*} \right\rangle \nonumber \\
		& = - \left\langle \lambda_{*} , Ax \left( t \right) - b \right\rangle \nonumber \\
		& \geq - \left\lVert \lambda_{*} \right\rVert \left\lVert Ax \left( t \right) - b \right\rVert
		\geq - \dfrac{\left\lVert \lambda_{*} \right\rVert \Cbnd}{\theta^{2} t^{2}}. \label{rate:lower}
		\end{align}
		By combining \eqref{rate:upper} and \eqref{rate:lower} we obtain the desired statement.
		\qedhere
	\end{enumerate}
\end{proof}

\begin{rmk}
A few remarks comparing our convergence rate results with the ones reported in  \cite{Attouch-Chbani-Fadili-Riahi,He-Hu-Fang,Zeng-Lei-Chen} are in order.
	\begin{enumerate}
		\item [$\bullet$] \emph{Primal-dual gap}: Relation \eqref{rate:gap} guarantees a convergence rate for the primal-dual gap of
		\begin{equation*}
		\Lag \left( x \left( t \right) , \lambda_{*} \right) - \Lag \left( x_{*} , \lambda(t) \right) = \bO \left( \dfrac{1}{t^{2}} \right)  \quad \mbox{as} \quad t \rightarrow +\infty,
		\end{equation*}
which can be equivalently written as
		\begin{equation*}
		\Lag \left( x \left( t \right) , \lambda_{*} \right) - \Lag \left( x_{*} , \lambda_{*} \right) = \bO \left( \dfrac{1}{t^{2}} \right) \quad \mbox{as} \quad t \rightarrow +\infty.
		\end{equation*}
The primal-dual gap convergence rate stated in this form has been reported in \cite{Attouch-Chbani-Fadili-Riahi,He-Hu-Fang,Zeng-Lei-Chen}.
\item [$\bullet$] \emph{Feasibility measure}: Relation {\eqref{rate:fea}} guarantees a convergence rate for the feasibility measure of
		\begin{equation*}
		\left\lVert A x \left( t \right) - b \right\rVert = \bO \left( \dfrac{1}{t^{2}} \right ) \quad \mbox{as} \quad t \rightarrow +\infty,
		\end{equation*}
In \cite{Attouch-Chbani-Fadili-Riahi,He-Hu-Fang,Zeng-Lei-Chen}, the feasibility measure $\left\lVert A x \left( t \right) - b \right\rVert$ is reported to have a convergence rate of $\bO \left( 1/t \right)$ as $t \rightarrow +\infty$.
		
\item [$\bullet$] \emph{Objective function value}: The upper bound we report for the objective function value in \eqref{rate:fun} matches the one from  \cite{Attouch-Chbani-Fadili-Riahi}, while our lower bound, which is of order $\frac{1}{t^2}$, outperforms the one reported in  \cite{Attouch-Chbani-Fadili-Riahi}, which is of order $\frac{1}{t}$. In \cite{He-Hu-Fang, Zeng-Lei-Chen} no convergence rates for the objective function value are provided.
	\end{enumerate}
\end{rmk}

\section{Weak convergence of the trajectory to a primal-dual optimal solution}
\label{sec:con}

The study of the convergence of the trajectory will be made in the following setting, which will be assumed to be fulfilled throughout the whole section.

\begin{mdframed}
	\begin{assume}
		\label{assume:strict}
		Suppose that $\nabla f$ is $\ell-$Lipschitz continuous and $\alpha, \beta$ and $\theta$ in \eqref{ds:PD-AVD} satisfy
		\begin{equation*}
		\alpha > 3 , \quad {\beta \geq 0}
		\quad \textrm{ and } \quad
		\dfrac{1}{2} > \theta > \dfrac{1}{\alpha - 1}.
		\end{equation*}
	\end{assume}
\end{mdframed}

For the beginning we will prove that in the setting of Assumption \ref{assume:strict} the dynamical system \eqref{ds:PD-AVD} has a unique global twice continuously differentiable solution.

\begin{thm} For every initial condition
	$$\Bigl( x \left( t_{0} \right) , \lambda \left( t_{0} \right) \Bigr) := \Bigl( x_{0} , \lambda_{0} \Bigr) \in \sX \times \sY 
	\quad \textrm{ and } \quad 
	\Bigl( \dot{x} \left( t_{0} \right) , \dot{\lambda} \left( t_{0} \right) \Bigr) := \Bigl(\dot x_{0} , \dot \lambda_{0}\Bigr) \in \sX \times \sY$$
the dynamical system \eqref{ds:PD-AVD} has a unique global twice continously differentiable solution $\left( x , \lambda \right) \colon \left[ t_{0} , + \infty \right) \to \sX \times \sY$.
\end{thm}
\begin{proof}
We observe that $\left( x , \lambda \right) \colon \left[ t_{0} , + \infty \right) \to \sX \times \sY$ is a solution of \eqref{ds:PD-AVD} if and only if $\left( x , \lambda , y , \nu \right) \colon \left[ t_{0} , + \infty \right) \to \sX \times \sY \times \sX \times \sY$ is a solution of the first-order dynamical system
\begin{equation}
\begin{dcases}
\dot{x} \left( t \right) 	 = y \left( t \right) \\
\dot{\lambda} \left( t \right) = \nu \left( t \right) \\
\dot{y} \left( t \right)		= - \dfrac{\alpha}{t} y \left( t \right) - \nabla f \left( x \left( t \right) \right) - A^* \left( \lambda \left( t \right) + \theta t \nu \left( t \right) \right) - \beta A^* \Bigl( Ax \left( t \right) - b \Bigr) \\
\dot{\nu} \left( t \right)  	= - \dfrac{\alpha}{t} \nu \left( t \right) + \Bigl( A \bigl( x \left( t \right) + \theta t \dot{x} \left( t \right) \bigr) - b \Bigr) \\
\Bigl( x \left( t_{0} \right) , \lambda \left( t_{0} \right) , y \left( t_{0} \right) , \nu \left( t_{0} \right) \Bigr) 			
= \Bigl( x_{0} , \lambda_{0} , \dot{x}_{0} , \dot{\lambda}_{0} \Bigr)
\end{dcases} . \label{e-u:ds}
\end{equation}
For $F \colon \left[ t_{0} , + \infty \right) \times \sX \times \sY \times \sX \times \sY \to \sX \times \sY \times \sX \times \sY$ by
\begin{align*}
& F \left( t , z , \zeta , u , \rho \right) := \nonumber \\
& \begin{pmatrix}
u , 
\rho , 
- \dfrac{\alpha}{t} u - \nabla f \left( z \right) - A^* \left( \zeta + \theta t \rho \right) - \beta A^* \Bigl( Az - b \Bigr) ,
- \dfrac{\alpha}{t} \rho + \Bigl( A \bigl( z + \theta t u \bigr) - b \Bigr)
\end{pmatrix} ,
\end{align*}
\eqref{e-u:ds} can be equivalently written as
\begin{equation*}
\begin{dcases}
\begin{pmatrix}
\dot{x} \left( t \right) , \dot{\lambda} \left( t \right) , \dot{y} \left( t \right) , \dot{\nu} \left( t \right)
\end{pmatrix} 
= F \begin{pmatrix}
t , x \left( t \right) , \lambda \left( t \right) , y \left( t \right) , \nu \left( t \right)
\end{pmatrix} \\
\Bigl( x \left( t_{0} \right) , \lambda \left( t_{0} \right) , y \left( t_{0} \right) , \nu \left( t_{0} \right) \Bigr) 			
= \Bigl( x_{0} , \lambda_{0} , \dot{x}_{0} , \dot{\lambda}_{0} \Bigr)
\end{dcases}.
\end{equation*}

Next we will show that $F$ is Lipschitz continuous on bounded sets and chose to this end arbitrary $t_{0} \leq t_{1} < t_{2} < + \infty$ and $\delta > 0$. For
\begin{align*}
\left( t , z , \zeta , u , \rho \right), \left( \bt , \bz , \bzeta , \bu , \brho \right) \in & \left[t_{1} , t_{2} \right] \times \sB \left( 0 ; \delta \right) \times \sB \left( 0 ; \delta \right) \times \sB \left( 0 ; \delta \right) \times \sB \left( 0 ; \delta \right) \\
\subseteq & \left[ t_{0} , + \infty \right) \times \sX \times \sY \times \sX \times \sY,
\end{align*}
we have
\begin{align*}
& \left\lVert F \left( t , z , \zeta , u , \rho \right) - F \left( \bt , \bz , \bzeta , \bu , \brho \right) \right\rVert \nonumber \\
\leq & \left\lVert u - \bu \right\rVert
+ \left\lVert \rho - \brho \right\rVert + \nonumber \\
& \left\lVert \dfrac{\alpha}{t} u - \dfrac{\alpha}{\bt} \bu + \nabla f \left( z \right) - \nabla f \left( \bz \right) + A^* \left( \zeta - \bzeta + \theta \left( t \rho - \bt \brho \right) \right) + \beta A^* A \left( z - \bz \right) \right\rVert +  \nonumber \\
& \left\lVert \dfrac{\alpha}{t} \rho - \dfrac{\alpha}{\bt} \brho - \Bigl( A \bigl( z - \bz + \theta \left( t u - \bt \bu \right) \bigr) \Bigr) \right\rVert \nonumber \\
\leq & \left\lVert u - \bu \right\rVert
+ \left\lVert \rho - \brho \right\rVert 
+ \left(\beta \left\lVert A \right\rVert^2 + \|A\| + \ell \right)  \left\lVert z - \bz \right\rVert 
+ \left\lVert A \right\rVert \left\lVert \zeta - \bzeta \right\rVert +  \nonumber \\
& \alpha \left\lVert \dfrac{1}{t} u - \dfrac{1}{\bt} \bu \right\rVert
+ \theta \left\lVert A \right\rVert \left\lVert t \rho - \bt \brho \right\rVert 
+ \alpha \left\lVert \dfrac{1}{t} \rho - \dfrac{1}{\bt} \brho \right\rVert
+ \theta \left\lVert A \right\rVert \left\lVert t u - \bt \bu \right\rVert \nonumber \\
\leq & \left( 1 + \dfrac{\alpha}{t} + \theta t \left\lVert A \right\rVert \right) \left\lVert u - \bu \right\rVert
+ \left( 1 +  \dfrac{\alpha}{t} + \theta t \left\lVert A \right\rVert \right) \left\lVert \rho - \brho \right\rVert 
+ \left(\beta \left\lVert A \right\rVert^2 + \|A\| + \ell \right) \left\lVert z - \bz \right\rVert + \nonumber \\
& \left\lVert A \right\rVert \left\lVert \zeta - \bzeta \right\rVert
+ \alpha \left( \left\lVert \bu \right\rVert + \left\lVert \brho \right\rVert \right) \left\lvert \dfrac{1}{t} - \dfrac{1}{\bt} \right\rvert
+ \theta \left\lVert A \right\rVert \left( \left\lVert \bu \right\rVert + \left\lVert \brho \right\rVert \right) \left\lvert t - \bt \right\rvert \nonumber \\
\leq & \left( 1 + \dfrac{\alpha}{t_1} + \theta t_2 \left\lVert A \right\rVert  \right) \left( \left\lVert u - \bu \right\rVert  + \left\lVert \rho - \brho \right\rVert \right) + \left(\beta \left\lVert A \right\rVert^2 + \|A\| + \ell \right) \left\lVert z - \bz \right\rVert + \nonumber \\
& \left\lVert A \right\rVert \left\lVert \zeta - \bzeta \right\rVert
+ 2 \delta \left( \dfrac{\alpha}{t_{1}^{2}} + \theta \left\lVert A \right\rVert \right) \left\lvert t - \bt \right\rvert .
\end{align*}
Consequently, 
\begin{equation*}
\left\lVert F \left( t , z , \zeta , u , \rho \right) - F \left( \bt , \bz , \bzeta , \bu , \brho \right) \right\rVert \leq L_{F} \left\lVert \left( t , z , \zeta , u , \rho \right) - \left( \bt , \bz , \bzeta , \bu , \brho \right) \right\rVert ,
\end{equation*}
where
\begin{equation*}
L_{F} := \sqrt{2 \left( 1 + \dfrac{\alpha}{t_1} + \theta t_2 \left\lVert A \right\rVert  \right)^2 + \left(\beta \left\lVert A \right\rVert^2 + \|A\| + \ell \right)^2 + \left\lVert A \right\rVert ^{2} + 4 \delta ^{2} \left( \dfrac{\alpha}{t_{1}^{2}} + \theta \left\lVert A \right\rVert \right) ^{2}} .
\end{equation*}

Since $F$ is Lipschitz continuous on bounded sets and continuously differentiable, the local existence and uniqueness theorem (see, for instance, \cite[Theorems 46.2 and 46.3]{Sell-You}) allows us to conclude that there exists a unique solution $\left( x , \lambda , y , \nu \right) \in \sX \times \sY \times \sX \times \sY$ of \eqref{e-u:ds} defined on a maximally interval $\left[ t_{0} , \Tm \right)$ where $t_{0} < \Tm \leq + \infty$. Furthermore, either
\begin{equation*}
\Tm = + \infty 
\qquad \textrm{ or } \qquad 
\lim\limits_{t \to \Tm} \left\lVert \left( x \left( t \right) , \lambda \left( t \right) , y \left( t \right) , \nu \left( t \right) \right) \right\rVert = + \infty.
\end{equation*}
We will prove that $\Tm = + \infty$.

Let $\left( x_{*} , \lambda_{*} \right) \in \sol$. According to Lemma \ref{lem:dec} we have for every $t_{0} \leq t < \Tm$
\begin{align*}
\dfrac{d}{dt} \E_{x_{*} , \lambda_{*}} \left( t \right) 
& \leq \left( 2 \theta - 1 \right) \theta t \G_{\beta} \Bigl( \bigl( x \left( t \right) , \lambda \left( t \right) \bigr) \Big\vert \left( x_{*} , \lambda_{*} \right) \Bigr) - \dfrac{\beta \theta t}{2} \left\lVert Ax \left( t \right) - b \right\rVert ^{2}  - \xi \theta t \left\lVert \left( \dot{x} \left( t \right) , \dot{\lambda} \left( t \right) \right) \right\rVert ^{2} \\
& \leq 0 .
\end{align*}
From here it follows, as in Theorem \ref{thm:bnd} (see \eqref{bnd:tra-sol} and \eqref{bnd:der}), that for every $t_{0} \leq t < \Tm$ it holds
\begin{align*}
\left\lVert \left( x \left( t \right) , \lambda \left( t \right) \right) \right\rVert
& \leq \left\lVert \bigl( x \left( t \right) , \lambda \left( t \right) \bigr) - \left( x_{*} , \lambda_{*} \right) \right\rVert 
+ \left\lVert \left( x_{*} , \lambda_{*} \right) \right\rVert
\leq \sqrt{\dfrac{2 \E_{x_{*} , \lambda_{*}} \left( t_{0} \right)}{\xi}} + \left\lVert \left( x_{*} , \lambda_{*} \right) \right\rVert , \nonumber \\
\left\lVert (\dot{x} \left( t \right) , \dot{\lambda} \left( t \right) )\right\rVert
& = \left\lVert (y \left( t \right) , \nu \left( t \right) )\right\rVert
\leq \dfrac{1}{\theta t} \left( \dfrac{1}{\sqrt{\xi}} + 1 \right) \sqrt{2 \E_{x_{*} , \lambda_{*}} \left( t_{0} \right)}
\leq \dfrac{1}{\theta t_{0}} \left( \dfrac{1}{\sqrt{\xi}} + 1 \right) \sqrt{2 \E_{x_{*} , \lambda_{*}} \left( t_{0} \right)} .
\end{align*}
Consequently, $ t \mapsto \left( x \left( t \right) , \lambda \left( t \right) , y \left( t \right) , \nu \left( t \right) \right)$ is bounded on $[t_0, \Tm)$, which means that the limit $\lim\limits_{t \to \Tm} \left\lVert \left( x \left( t \right) , \lambda \left( t \right) , y \left( t \right) , \nu \left( t \right) \right) \right\rVert$ cannot be $+ \infty$. In conclusion, $\Tm = + \infty$, which completes the proof.
\end{proof}

We start the convergence analysis of the trajectory with the proof of two important integrability results, whereby we notice that that statement \eqref{bnd:int:fea} only implies \eqref{int:fea} if $\beta > 0$.
\begin{prop}
	\label{prop:int}
	Let $\left( x , \lambda \right) \colon \left[ t_{0} , + \infty \right) \to \sX \times \sY$ be a solution of \eqref{ds:PD-AVD} and $\left( x_{*} , \lambda_{*} \right) \in \sol$. 
	Then it holds
	\begin{equation}
	\label{int:grad}
	\int_{t_{0}}^{+ \infty} t \left\lVert \nabla f \left( x \left( t \right) \right) - \nabla f \left( x_{*} \right) \right\rVert ^{2} dt < + \infty
	\end{equation}
	and
	\begin{equation}
	\label{int:fea}
	\int_{t_{0}}^{+ \infty} t \left\lVert A x \left( t \right) - b \right\rVert ^{2} dt < + \infty .
	\end{equation}
\end{prop}
\begin{proof}
The determinant role in the proof is the fact that, for every $t \geq t_0$, as $\nabla f$ is $\ell-$ Lipschitz continuous, relation \eqref{dec:conv} in the proof of Lemma \ref{lem:dec} can be sharpened thanks to \eqref{pre:f-bound} to
\begin{align*}
	& \ - \left\langle \nabla \G_{\beta} \Bigl( \bigl( x \left( t \right) , \lambda \left( t \right) \bigr) \Big\vert \left( x_{*} , \lambda_{*}  \right) \Bigr) , \bigl( x \left( t \right) , \lambda \left( t \right) \bigr) - \left( x_{*} , \lambda_{*}  \right) \right\rangle \nonumber \\
	= & \ \left\langle \nabla f \left( x \left( t \right) \right) , x_{*} - x \left( t \right) \right\rangle 
	+ \left\langle A^* \lambda_{*} , x_{*} - x \left( t \right) \right\rangle
	+ \beta \left\langle A^* \left( Ax \left( t \right) - b \right) , x_{*} - x \left( t \right) \right\rangle \nonumber \\
	\leq 	& \ - \left( f \left( x \left( t \right) \right) - f \left( x_{*} \right) \right) - \dfrac{1}{2 \ell} \left\lVert \nabla f \left( x \left( t \right) \right) - \nabla f \left( x_{*} \right) \right\rVert ^{2} - \left\langle \lambda_{*} , Ax \left( t \right) - b \right\rangle - \beta \left\lVert Ax \left( t \right) - b \right\rVert ^{2} \nonumber \\
	= & \ - \G_{\beta} \bigl( \bigl( x \left( t \right) , \lambda \left( t \right) \bigr) \vert \left( x_{*} , \lambda_{*} \right) \bigr) - \dfrac{1}{2 \ell} \left\lVert \nabla f \left( x \left( t \right) \right) - \nabla f \left( x_{*} \right) \right\rVert ^{2} - \dfrac{\beta}{2} \left\lVert Ax \left( t \right) - b \right\rVert ^{2} .
	\end{align*}
	Consequently, by combining this inequality with \eqref{dec:pre}, it yields for every $t \geq t_0$
	\begin{align*}
	\dfrac{d}{dt} \E_{x_{*} , \lambda_{*}} \left( t \right) 
	\leq & \left( 2 \theta - 1 \right) \theta t \G_{\beta} \Bigl( \bigl( x \left( t \right) , \lambda \left( t \right) \bigr) \Big\vert \left( x_{*} , \lambda_{*} \right) \Bigr) 
	- \xi \theta t \left\lVert \left( \dot{x} \left( t \right) , \dot{\lambda} \left( t \right) \right) \right\rVert ^{2} \nonumber \\
	& - \dfrac{\theta t}{2 \ell} \left\lVert \nabla f \left( x \left( t \right) \right) - \nabla f \left( x_{*} \right) \right\rVert ^{2} - \dfrac{\theta \beta t}{2} \left\lVert Ax \left( t \right) - b \right\rVert ^{2} \nonumber \\
	 \leq & - \dfrac{\theta t}{2 \ell} \left\lVert \nabla f \left( x \left( t \right) \right) - \nabla f \left( x_{*} \right) \right\rVert ^{2} .
	\end{align*}
This leads  by integration to \eqref{int:grad}. 

On the other hand, it follows from \eqref{rate:fea} that
	\begin{equation*}
	\int_{t_{0}}^{+ \infty} t \left\lVert A x \left( t \right) - b \right\rVert ^{2} dt \leq \dfrac{\Cbnd^{2}}{\theta^{4}} \int_{t_{0}}^{+ \infty} \dfrac{1}{t^{3}} < + \infty ,
	\end{equation*}
	and the proof is complete.
\end{proof}

Now we
define, for a given primal-dual optimal solution $\left( x_{*} , \lambda_{*} \right) \in \sol$, the following two mappings on $[t_0,+\infty)$
\begin{align*}
W \left( t \right)	
& := \Lb \left( x \left( t \right) , \lambda_{*} \right) - \Lb \left( x_{*} , \lambda \left( t \right) \right) + \dfrac{1}{2} \left\lVert \left( \dot{x} \left( t \right) , \dot{\lambda} \left( t \right) \right) \right\rVert ^{2} \geq 0 \\
\varphi \left( t \right)
& := \dfrac{1}{2} \left\lVert \bigl( x \left( t \right) , \lambda \left( t \right) \bigr) - \left( x_{*} , \lambda_{*} \right) \right\rVert ^{2} \geq 0. 
\end{align*}

\begin{lem}
	\label{lem:est-1}
	Let $\left( x , \lambda \right) \colon \left[ t_{0} , + \infty \right) \to \sX \times \sY$ be a solution of \eqref{ds:PD-AVD} and $\left( x_{*} , \lambda_{*} \right) \in \sol$. The following inequality holds for every $t \geq t_{0}$:
	\begin{equation}
	\label{est:pre}
	\ddot{\varphi} \left( t \right) + \dfrac{\alpha}{t} \dot{\varphi} \left( t \right) + \theta t \dot{W} \left( t \right) 
	+ \dfrac{1}{2 \ell} \left\lVert \nabla f \left( x \left( t \right) \right) - \nabla f \left( x_{*} \right) \right\rVert ^{2} + \dfrac{\beta}{2} \left\lVert Ax \left( t \right) - b \right\rVert ^{2} \leq 0 . 
	\end{equation}
\end{lem}
\begin{proof}
Let $t \geq t_{0}$ be fixed. The time derivative of $W$ reads
\begin{align}
\dot{W} \left( t \right)
= & \left\langle \nabla_{x} \Lb \left( x \left( t \right) , \lambda_{*} \right) , \dot{x} \left( t \right) \right\rangle 
+ \left\langle \ddot{x} \left( t \right) , \dot{x} \left( t \right) \right\rangle 
+ \left\langle \ddot{\lambda} \left( t \right) , \dot{\lambda} \left( t \right) \right\rangle \nonumber \\
= & \left\langle \nabla_{x} \Lb \left( x \left( t \right) , \lambda \left( t \right) + \theta t \dot{\lambda} \left( t \right) \right) , \dot{x} \left( t \right) \right\rangle 	
+ \left\langle \ddot{x} \left( t \right) , \dot{x} \left( t \right) \right\rangle - \left\langle \lambda \left( t \right) - \lambda_{*} + \theta t \dot{\lambda} \left( t \right) , A \dot{x} \left( t \right) \right\rangle  \nonumber \\
& - \left\langle \nabla_{\lambda} \Lb \left( x \left( t \right) + \theta t \dot{x} \left( t \right) , \lambda \left( t \right) \right) , \dot{\lambda} \left( t \right) \right\rangle
+ \left\langle \ddot{\lambda} \left( t \right) , \dot{\lambda} \left( t \right) \right\rangle + \left\langle Ax \left( t \right) - b + \theta t A \dot{x} \left( t \right) , \dot{\lambda} \left( t \right) \right\rangle \nonumber \\
= & - \dfrac{\alpha}{t} \left\lVert \dot{x} \left( t \right) \right\rVert ^{2} 
- \dfrac{\alpha}{t} \left\lVert \dot{\lambda} \left( t \right) \right\rVert ^{2}
- \left\langle \lambda \left( t \right) - \lambda_{*} , A \dot{x} \left( t \right) \right\rangle
+ \left\langle Ax \left( t \right) - b , \dot{\lambda} \left( t \right) \right\rangle .
\label{der:dW}
\end{align}

On the one hand, by the chain rule, we have
\begin{align*}
\dot{\varphi} \left( t \right)
& = \left\langle x \left( t \right) - x_{*} , \dot{x} \left( t \right) \right\rangle + \left\langle \lambda \left( t \right) - \lambda_{*} , \dot{\lambda} \left( t \right) \right\rangle , \\
\ddot{\varphi} \left( t \right)
& = \left\langle x \left( t \right) - x_{*} , \ddot{x} \left( t \right) \right\rangle + \left\lVert \dot{x} \left( t \right) \right\rVert ^{2} + \left\langle \lambda \left( t \right) - \lambda_{*} , \ddot{\lambda} \left( t \right) \right\rangle + \left\lVert \dot{\lambda} \left( t \right) \right\rVert ^{2} . 
\end{align*}
By combining these relations and using that $Ax_* =b$, we get
\begin{align}
\ddot{\varphi} \left( t \right) + \dfrac{\alpha}{t} \dot{\varphi} \left( t \right) 
= & \left\langle x \left( t \right) - x_{*} , \ddot{x} \left( t \right) + \dfrac{\alpha}{t} \dot{x} \left( t \right) \right\rangle + \left\langle \lambda \left( t \right) - \lambda_{*} , \ddot{\lambda} \left( t \right) + \dfrac{\alpha}{t} \dot{\lambda} \left( t \right) \right\rangle  + \left\lVert \dot{x} \left( t \right) \right\rVert ^{2} + \left\lVert \dot{\lambda} \left( t \right) \right\rVert ^{2} \nonumber \\
= & - \left\langle x \left( t \right) - x_{*} , \nabla_{x} \Lb \left( x \left( t \right) , \lambda \left( t \right) + \theta t \dot{\lambda} \left( t \right) \right) \right\rangle \nonumber \\
& + \left\langle \lambda \left( t \right) - \lambda_{*} , \nabla_{\lambda} \Lb \left( x \left( t \right) + \theta t \dot{x} \left( t \right) , \lambda \left( t \right) \right) \right\rangle + \left\lVert \dot{x} \left( t \right) \right\rVert ^{2} + \left\lVert \dot{\lambda} \left( t \right) \right\rVert ^{2} \nonumber \\ 
= & - \left\langle x \left( t \right) - x_{*} , \nabla_{x} \Lb \left( x \left( t \right) , \lambda_{*} \right) \right\rangle 
- \left\langle Ax \left( t \right) - b , \lambda \left( t \right) - \lambda_{*} + \theta t \dot{\lambda} \left( t \right) \right\rangle \nonumber \\
& + \left\langle \lambda \left( t \right) - \lambda_{*} , Ax \left( t \right) - b + \theta t A \dot{x} \left( t \right) \right\rangle 
+ \left\lVert \dot{x} \left( t \right) \right\rVert ^{2} + \left\lVert \dot{\lambda} \left( t \right) \right\rVert ^{2} . \label{der:est} 
\end{align}
By exploiting the Lipschitz continuity of $\nabla f$ (see \eqref{pre:f-bound}) and using again that $Ax_* =b$, we obtain the following estimate
\begin{align*}
& - \left\langle x \left( t \right) - x_{*} , \nabla_{x} \Lb \left( x \left( t \right) , \lambda_{*} \right) \right\rangle \nonumber \\
= \ 	& - \left\langle x \left( t \right) - x_{*} , \nabla f \left( x \left( t \right) \right) \right\rangle - \left\langle x \left( t \right) - x_{*} , A^* \lambda_{*} \right\rangle - \beta \left\lVert Ax \left( t \right) - b \right\rVert ^{2} \nonumber \\
\leq \ 	& - \left( f \left( x \left( t \right) \right) - f_{*} \right) - \dfrac{1}{2 \ell} \left\lVert \nabla f \left( x \left( t \right) \right) - \nabla f \left( x_{*} \right) \right\rVert ^{2} - \left\langle \lambda_{*} , Ax \left( t \right) - b \right\rangle - \beta \left\lVert Ax \left( t \right) - b \right\rVert ^{2} \nonumber \\
= \ 	& - \Bigl( \Lb \left( x \left( t \right) , \lambda_{*} \right) - \Lb \left( x_{*} , \lambda \left( t \right) \right) \Bigr) - \dfrac{1}{2 \ell} \left\lVert \nabla f \left( x \left( t \right) \right) - \nabla f \left( x_{*} \right) \right\rVert ^{2} - \dfrac{\beta}{2} \left\lVert Ax \left( t \right) - b \right\rVert ^{2},
\end{align*}
which, in combination with \eqref{der:est}, leads to
\begin{align}
\ddot{\varphi} \left( t \right) + \dfrac{\alpha}{t} \dot{\varphi} \left( t \right) 
= & - \left\langle x \left( t \right) - x_{*} , \nabla_{x} \Lb \left( x \left( t \right) , \lambda_{*} \right) \right\rangle 
- \theta t \left\langle Ax \left( t \right) - b , \dot{\lambda} \left( t \right) \right\rangle \nonumber \\
& + \theta t \left\langle \lambda \left( t \right) - \lambda_{*} , A \dot{x} \left( t \right) \right\rangle 
+ \left\lVert \dot{x} \left( t \right) \right\rVert ^{2} + \left\lVert \dot{\lambda} \left( t \right) \right\rVert ^{2} \nonumber \\
 \leq & - \Bigl( \Lb \left( x \left( t \right) , \lambda_{*} \right) - \Lb \left( x_{*} , \lambda \left( t \right) \right) \Bigr) 
- \theta t \left\langle Ax \left( t \right) - b , \dot{\lambda} \left( t \right) \right\rangle \nonumber \\
& + \theta t \left\langle \lambda \left( t \right) - \lambda_{*} , A \dot{x} \left( t \right) \right\rangle 
+ \left\lVert \dot{x} \left( t \right) \right\rVert ^{2} + \left\lVert \dot{\lambda} \left( t \right) \right\rVert ^{2} \nonumber \\
& - \dfrac{1}{2 \ell} \left\lVert \nabla f \left( x \left( t \right) \right) - \nabla f \left( x_{*} \right) \right\rVert ^{2} - \dfrac{\beta}{2} \left\lVert Ax \left( t \right) - b \right\rVert ^{2} . \label{est:tight}
\end{align}	
Multiplying \eqref{der:dW} by $\theta t > 0$ then summing the result to \eqref{est:tight} yields
\begin{align*}
& \ddot{\varphi} \left( t \right) + \dfrac{\alpha}{t} \dot{\varphi} \left( t \right) + \theta t \dot{W} \left( t \right) 
+ \dfrac{1}{2 \ell} \left\lVert \nabla f \left( x \left( t \right) \right) - \nabla f \left( x_{*} \right) \right\rVert ^{2} + \dfrac{\beta}{2} \left\lVert Ax \left( t \right) - b \right\rVert ^{2} \nonumber \\
\leq & - \Bigl( \Lb \left( x \left( t \right) , \lambda_{*} \right) - \Lb \left( x_{*} , \lambda \left( t \right) \right) \Bigr) 
+ \left(1- \theta \alpha \right) \left\lVert \left( \dot{x} \left( t \right) , \dot{\lambda} \left( t \right) \right) \right\rVert ^{2} \nonumber\\
 \leq & \ 0 , 
\end{align*}
since $\theta > \dfrac{1}{\alpha - 1} > \dfrac{1}{\alpha}$.
\end{proof}

The following result provides one of the two statements of the Opial Lemma (see Lemma \ref{lem:Opial}) which we will use to prove weak convergence of the trajectory.

\begin{lem}\label{lem:lim-phi}
Let $\left( x , \lambda \right) \colon \left[ t_{0} , + \infty \right) \to \sX \times \sY$ be a solution of \eqref{ds:PD-AVD} and $\left( x_{*} , \lambda_{*} \right) \in \sol$. Then the positive part $\left[ \dot{\varphi} \right] _{+}$ of $\dot \varphi$ belongs to $\sL^{1} \left([ t_{0} , + \infty) \right)$ and the limit $\lim\limits_{t \to + \infty} \varphi \left( t \right) \in \sR$ exists.
\end{lem}
\begin{proof}
Multiplying inequality \eqref{est:pre}  by $t$ and adding $\theta(\alpha+1)tW(t)$ to its both sides, we obtain for every $t \geq t_0$
\begin{align}
t \ddot{\varphi} \left( t \right) + \alpha \dot{\varphi} \left( t \right)
+ \theta \left( t^{2} \dot{W} \left( t \right) + \left( \alpha + 1 \right) t W \left( t \right) \right) 
\leq \theta \left( \alpha + 1 \right) t W \left( t \right). \label{lim:t} 
\end{align}
Multiplying further \eqref{lim:t} by $t^{\alpha - 1}$, it yields for every $t \geq t_0$
\begin{equation}
\label{lim:inq}
\dfrac{d}{dt} \left( t^{\alpha} \dot{\varphi} \left( t \right) \right) + \theta \dfrac{d}{dt} \left( t^{\alpha + 1} W \left( t \right) \right) \leq \theta(\alpha+1) t^{\alpha} W \left( t \right) .
\end{equation}
As $1-2 \theta >0$ and $\xi = \theta \alpha - \theta - 1 >0$, it follows from \eqref{bnd:int:Lag} and \eqref{bnd:int:vel} in Theorem \ref{thm:bnd} that $t \mapsto tW(t)$ belongs to $\sL^{1} \left([ t_{0} , + \infty) \right)$.

After integration we obtain from \eqref{lim:inq} that for every $t \geq t_0$
\begin{equation*}
t^{\alpha} \dot{\varphi} \left( t \right) - t_{0}^{\alpha} \dot{\varphi} \left( t_{0} \right) + \theta \left( t^{\alpha + 1} W \left( t \right) - t_{0}^{\alpha + 1} W \left( t_{0} \right) \right)
\leq \theta(\alpha +1) \int_{t_{0}}^{t} s^{\alpha} W \left( s \right) ds
\end{equation*}
which yields
\begin{equation*}
\dot{\varphi} \left( t \right) \leq \dfrac{1}{t^{\alpha}} \left( t_{0}^{\alpha} \left\lvert \dot{\varphi} \left( t_{0} \right) \right\rvert + \theta t_{0}^{\alpha + 1} W \left( t_{0} \right) \right)
+ \dfrac{\theta(\alpha +1)}{t^{\alpha}} \int_{t_{0}}^{t} s^{\alpha} W \left( s \right) ds .
\end{equation*}
We set
\begin{equation*}
\Cint := t_{0}^{\alpha} \left\lvert \dot{\varphi} \left( t_{0} \right) \right\rvert + \theta t_{0}^{\alpha + 1} W \left( t_{0} \right) \geq 0
\end{equation*}
and obtain further that for every $t \geq t_0$
\begin{equation*}
\left[ \dot{\varphi} \left( t \right) \right] _{+} 
\leq \dfrac{\Cint}{t^{\alpha}} + \dfrac{\theta(\alpha +1)}{t^{\alpha}} \int_{t_{0}}^{t} s^{\alpha} W \left( s \right) ds
\end{equation*}
and after integration
\begin{equation*}
\int_{t_{0}}^{+ \infty} \left[ \dot{\varphi} \left( t \right) \right] _{+} dt
\leq \Cint \int_{t_{0}}^{+ \infty} \dfrac{1}{t^{\alpha}} dt
+ \theta(\alpha +1) \int_{t_{0}}^{+ \infty} \dfrac{1}{t^{\alpha}} \left( \int_{t_{0}}^{t} s^{\alpha} W \left( s \right) ds \right) dt .
\end{equation*}
We have
\begin{equation*}
\int_{t_{0}}^{+ \infty} \dfrac{1}{t^{\alpha}} dt = \dfrac{1}{\left( \alpha - 1 \right) t_{0}^{\alpha - 1}}
\end{equation*}
and, by applying Lemma \ref{lem:est-int} with $h \left( t \right) := tW \left( t \right)$ and $r := + \infty$, 
\begin{align*}
\int_{t_{0}}^{+ \infty} \dfrac{1}{t^{\alpha}} \left( \int_{t_{0}}^{t} s^{\alpha} W \left( s \right) ds \right) dt
& = \dfrac{1}{\alpha - 1} \int_{t_{0}}^{+ \infty} tW \left( t \right) dt .
\end{align*}
Combining these relations we conclude that
\begin{equation*}
\int_{t_{0}}^{+ \infty} \left[ \dot{\varphi} \left( t \right) \right] _{+} dt
\leq \dfrac{\Cint}{\left( \alpha - 1 \right) t_{0}^{\alpha - 1}} + \dfrac{\theta(\alpha + 1)}{\alpha - 1} \int_{t_{0}}^{+ \infty} tW \left( t \right) dt < + \infty .
\end{equation*}
Finally, let $\psi \colon \left[ t_{0} , + \infty \right) \to \sR$ be the function defined by
\begin{equation*}
\psi \left( t \right) := \varphi \left( t \right) - \int_{t_{0}}^{t} \left[ \dot{\varphi} \left( s \right) \right] _{+} ds .
\end{equation*}
This function is nonincreasing and bounded from below, thus it has a finite limit as $t \rightarrow +\infty$. From here it yields that the limit
\begin{equation*}
\lim\limits_{t \to + \infty} \varphi \left( t \right)
= \lim\limits_{t \to + \infty} \psi \left( t \right) + \int_{t_{0}}^{+ \infty} \left[ \dot{\varphi} \left( s \right) \right] _{+} ds \in \sR
\end{equation*}
exists. 
\end{proof}

Next we will prove a number of results which will finally guarantee that the second assumption of the Opial Lemma is fulfilled, namely that every weak sequential cluster point of the trajectory $(x, \lambda)$ is an element of $\sol$.

\begin{lem}\label{lem:est-2}
Let $\left( x , \lambda \right) \colon \left[ t_{0} , + \infty \right) \to \sX \times \sY$ be a solution of \eqref{ds:PD-AVD} and $\left( x_{*} , \lambda_{*} \right) \in \sol$. The following inequality holds for every $t \geq t_{0}$:
	\begin{align*}	
	& \dfrac{\alpha}{t} \dfrac{d}{dt} \left\lVert \left( \dot{x} \left( t \right) , \dot{\lambda} \left( t \right) \right) \right\rVert ^{2}	
	+ \theta \dfrac{d}{dt} \left( t \left\lVert A^* \left( \lambda \left( t \right) - \lambda_{*} \right) \right\rVert ^{2} \right) + \left( 1 - \theta \right) \left\lVert A^* \left( \lambda \left( t \right) - \lambda_{*} \right) \right\rVert ^{2} \nonumber \\
	& + 2 \left\langle \ddot{x} \left( t \right) + \dfrac{\alpha}{t} \dot{x} \left( t \right) , A^* \left( \lambda \left( t \right) - \lambda_{*} \right) \right\rangle\\
	\leq  & \ {2 \left\lVert \nabla f \left( x \left( t \right) \right) - \nabla f \left( x_{*} \right) \right\rVert ^{2} + \left( 2 \beta^{2} \left\lVert A \right\rVert ^{2} + 1 \right) \left\lVert Ax \left( t \right) - b \right\rVert ^{2}} . 
	\end{align*}
\end{lem}
\begin{proof}
	Let $t \geq t_{0}$ be fixed. We have
	\begin{align}
	& \ \left\lVert \nabla f \left( x \left( t \right) \right) - \nabla f \left( x_{*} \right) + \beta A^* \left( Ax \left( t \right) - b \right) \right\rVert ^{2} \nonumber =  \left\lVert \ddot{x} \left( t \right) + \dfrac{\alpha}{t} \dot{x} \left( t \right) + A^* \left( \lambda \left( t \right) - \lambda_{*} + \theta t \dot{\lambda} \left( t \right) \right) \right\rVert ^{2} \nonumber \\
	= & \  \left\lVert \ddot{x} \left( t \right) + \dfrac{\alpha}{t} \dot{x} \left( t \right) \right\rVert ^{2} 
	+ \left\lVert A^* \left( \lambda \left( t \right) - \lambda_{*} + \theta t \dot{\lambda} \left( t \right) \right) \right\rVert ^{2} 
	+ 2 \left\langle \ddot{x} \left( t \right) + \dfrac{\alpha}{t} \dot{x} \left( t \right) , A^* \left( \lambda \left( t \right) - \lambda_{*} \right) \right\rangle \nonumber \\
	& + 2 \theta t \left\langle \ddot{x} \left( t \right) , A^* \dot{\lambda} \left( t \right) \right\rangle 
	+ 2 \alpha \theta \left\langle \dot{x} \left( t \right) , A^* \dot{\lambda} \left( t \right) \right\rangle \label{est:ds:x}
	\end{align}
	and
	\begin{align}
	& \left\lVert Ax \left( t \right) - b \right\rVert ^{2} = \left\lVert \ddot{\lambda} \left( t \right) + \dfrac{\alpha}{t} \dot{\lambda} \left( t \right) - \theta t A \dot{x} \left( t \right) \right\rVert ^{2} \nonumber \\
	 = & \left\lVert \ddot{\lambda} \left( t \right) + \dfrac{\alpha}{t} \dot{\lambda} \left( t \right) \right\rVert ^{2} 
	+ \theta^{2} t^{2} \left\lVert A \dot{x} \left( t \right) \right\rVert ^{2}
	- 2 \theta t \left\langle \ddot{\lambda} \left( t \right) , A \dot{x} \left( t \right) \right\rangle
	- 2 \alpha \theta \left\langle \dot{\lambda} \left( t \right) , A \dot{x} \left( t \right) \right\rangle \label{est:ds:lambda} .
	\end{align}	
	Summing \eqref{est:ds:x} and \eqref{est:ds:lambda}, we get
	\begin{align}
	& \ \left\lVert \nabla f \left( x \left( t \right) \right) - \nabla f \left( x_{*} \right) + \beta A^* \left( Ax \left( t \right) - b \right) \right\rVert ^{2} + \left\lVert Ax \left( t \right) - b \right\rVert ^{2} \nonumber \\
	= & \  \left\lVert \left (\ddot x(t), \ddot \lambda (t) \right) + \dfrac{\alpha}{t} \left (\dot x(t), \dot \lambda (t) \right)  \right\rVert ^{2} 
	+ \left\lVert A^* \left( \lambda \left( t \right) - \lambda_{*} + \theta t \dot{\lambda} \left( t \right) \right) \right\rVert ^{2} 
	+ \theta^{2} t^{2} \left\lVert A \dot{x} \left( t \right) \right\rVert ^{2} \nonumber \\
	& + 2 \theta t \left\langle \ddot{x} \left( t \right) , A^* \dot{\lambda} \left( t \right) \right\rangle - 2 \theta t \left\langle \ddot{\lambda} \left( t \right) , A \dot{x} \left( t \right) \right\rangle
	+ 2 \left\langle \ddot{x} \left( t \right) + \dfrac{\alpha}{t} \dot{x} \left( t \right) , A^* \left( \lambda \left( t \right) - \lambda_{*} \right) \right\rangle	. \label{est:ds:sum}
	\end{align}
We have
	\begin{align}
	& \ \theta^{2} t^{2} \left\lVert A \dot{x} \left( t \right) \right\rVert ^{2}
	+ 2 \theta t \left\langle \ddot{x} \left( t \right) , A^* \dot{\lambda} \left( t \right) \right\rangle 
	- 2 \theta t \left\langle \ddot{\lambda} \left( t \right) , A \dot{x} \left( t \right) \right\rangle \nonumber \\
	= & \ \theta^{2} t^{2} \left\lVert  \left( A^* \dot{\lambda} \left( t \right) , - A \dot{x} \left( t \right) \right) \right\rVert ^{2} - \theta^{2} t^{2} \left\lVert A^* \dot{\lambda} \left( t \right) \right\rVert ^{2} + 2 \theta t \left\langle \left (\ddot x(t), \ddot \lambda (t) \right),  \left( A^* \dot{\lambda} \left( t \right) , - A \dot{x} \left( t \right) \right)\right\rangle \nonumber \\
	= & \ - \left\lVert  \left (\ddot x(t), \ddot \lambda (t) \right) \right\rVert ^{2} + \left\lVert  \left (\ddot x(t), \ddot \lambda (t) \right) + \theta t \left( A^* \dot{\lambda} \left( t \right) , - A \dot{x} \left( t \right) \right)\right\rVert ^{2} - \theta^{2} t^{2} \left\lVert A^* \dot{\lambda} \left( t \right) \right\rVert ^{2} \nonumber \\
	\geq & \ - \left\lVert  \left (\ddot x(t), \ddot \lambda (t) \right) \right\rVert ^{2} - \theta^{2} t^{2} \left\lVert A^* \dot{\lambda} \left( t \right) \right\rVert ^{2} \label{est:ds:B}
	\end{align}
and
\begin{align}
	& \ \left\lVert \left (\ddot x(t), \ddot \lambda (t) \right) + \dfrac{\alpha}{t} \left( \dot{x} \left( t \right) , \dot{\lambda} \left( t \right) \right) \right\rVert ^{2} - \left\lVert  \left (\ddot x(t), \ddot \lambda (t) \right)  \right\rVert ^{2} \nonumber\\
	= & \ \dfrac{\alpha^{2}}{t^{2}} \left\lVert \left( \dot{x} \left( t \right) , \dot{\lambda} \left( t \right) \right) \right\rVert ^{2} + 2 \dfrac{\alpha}{t} \left\langle  \left (\ddot x(t), \ddot \lambda (t) \right), \left( \dot{x} \left( t \right) , \dot{\lambda} \left( t \right) \right) \right\rangle \geq \dfrac{\alpha}{t} \dfrac{d}{dt} \left\lVert \left( \dot{x} \left( t \right) , \dot{\lambda} \left( t \right) \right) \right\rVert ^{2}.\label{est:ds:w}
	\end{align}
In addition,
	\begin{align}
	& \ \left\lVert A^* \left( \lambda \left( t \right) - \lambda_{*} + \theta t \dot{\lambda} \left( t \right) \right) \right\rVert ^{2} - \theta^{2} t^{2} \left\lVert A^* \dot{\lambda} \left( t \right) \right\rVert ^{2}  \nonumber \\
	=  & \  \left\lVert A^* \left( \lambda \left( t \right) - \lambda_{*} \right) \right\rVert ^{2} 
	+ 2 \theta t \left\langle A A^* \left( \lambda \left( t \right) - \lambda_{*} \right) , \dot{\lambda} \left( t \right) \right\rangle \nonumber \\
	=  & \ {\left( 1 - \theta \right) \left\lVert A^* \left( \lambda \left( t \right) - \lambda_{*} \right) \right\rVert ^{2}}
	+ \theta \left\lVert A^* \left( \lambda \left( t \right) - \lambda_{*} \right) \right\rVert ^{2} 	
	+ \theta t \dfrac{d}{dt} \left\lVert A^* \left( \lambda \left( t \right) - \lambda_{*} \right) \right\rVert ^{2} \nonumber \\
	=  & \ {\left( 1 - \theta \right) \left\lVert A^* \left( \lambda \left( t \right) - \lambda_{*} \right) \right\rVert ^{2}}
	+ \theta \dfrac{d}{dt} \left( t \left\lVert A^* \left( \lambda \left( t \right) - \lambda_{*} \right) \right\rVert ^{2} \right) . \label{est:ds:A-lambda}
	\end{align}
	Hence, using \eqref{est:ds:B}, \eqref{est:ds:w} and \eqref{est:ds:A-lambda} in \eqref{est:ds:sum}  we obtain
	\begin{align*}
	& \ \left\lVert \nabla f \left( x \left( t \right) \right) - \nabla f \left( x_{*} \right) + \beta A^* \left( Ax \left( t \right) - b \right) \right\rVert ^{2} + \left\lVert Ax \left( t \right) - b \right\rVert ^{2} \nonumber \\
	\geq & \  \left\lVert \left (\ddot x(t), \ddot \lambda (t) \right) + \dfrac{\alpha}{t} \left( \dot{x} \left( t \right) , \dot{\lambda} \left( t \right) \right) \right\rVert ^{2} - \left\lVert  \left (\ddot x(t), \ddot \lambda (t) \right)  \right\rVert ^{2} \nonumber \\
	& \ + \left\lVert A^* \left( \lambda \left( t \right) - \lambda_{*} + \theta t \dot{\lambda} \left( t \right) \right) \right\rVert ^{2} - \theta^{2} t^{2} \left\lVert A^* \dot{\lambda} \left( t \right) \right\rVert ^{2} + 2 \left\langle \ddot{x} \left( t \right) + \dfrac{\alpha}{t} \dot{x} \left( t \right) , A^* \left( \lambda \left( t \right) - \lambda_{*} \right) \right\rangle \nonumber \\
	\geq 	& \ \dfrac{\alpha}{t}  \dfrac{d}{dt} \left\lVert \left( \dot{x} \left( t \right) , \dot{\lambda} \left( t \right) \right) \right\rVert ^{2} + \theta \dfrac{d}{dt} \left( t \left\lVert A^* \left( \lambda \left( t \right) - \lambda_{*} \right) \right\rVert ^{2} \right) + {\left( 1 - \theta \right) \left\lVert A^* \left( \lambda \left( t \right) - \lambda_{*} \right) \right\rVert ^{2}} \nonumber \\
	& \ + 2 \left\langle \ddot{x} \left( t \right) + \dfrac{\alpha}{t} \dot{x} \left( t \right) , A^* \left( \lambda \left( t \right) - \lambda_{*} \right) \right\rangle.
	\end{align*}
	Since
	\begin{align*}
	& \ \left\lVert \nabla f \left( x \left( t \right) \right) - \nabla f \left( x_{*} \right) + \beta A^* \left( Ax \left( t \right) - b \right) \right\rVert ^{2}\\ 
\leq &  \ 2 \left\lVert \nabla f \left( x \left( t \right) \right) - \nabla f \left( x_{*} \right) \right\rVert ^{2} + 2 \beta^{2} \left\lVert A \right\rVert ^{2} \left\lVert Ax \left( t \right) - b \right\rVert ^{2},
	\end{align*}
the conclusion follows.
\end{proof}
The following proposition provides a further important integrability result.
\begin{prop}
	\label{prop:int:t-lambda}
	Let $\left( x , \lambda \right) \colon \left[ t_{0} , + \infty \right) \to \sX \times \sY$ be a solution of \eqref{ds:PD-AVD} and $\left( x_{*} , \lambda_{*} \right) \in \sol$. Then it holds:
	\begin{equation*}
	\int_{t_{0}}^{+ \infty} t \left\lVert A^* \left( \lambda \left( t \right) - \lambda_{*} \right) \right\rVert ^{2} dt < + \infty .
	\end{equation*}
\end{prop}
\begin{proof}
From Lemma \ref{lem:est-1} and Lemma \ref{lem:est-2} we have for every $t \geq t_0$ that
	\begin{align}
	& \ \ddot{\varphi} \left( t \right) + \dfrac{\alpha}{t} \dot{\varphi} \left( t \right) + \theta t \dot{W} \left( t \right) 
	{+ \dfrac{\alpha}{t}} \dfrac{d}{dt} \left\lVert \left( \dot{x} \left( t \right) , \dot{\lambda} \left( t \right) \right) \right\rVert ^{2} \nonumber \\
	& \ {+ \theta} \dfrac{d}{dt} \left( t \left\lVert A^* \left( \lambda \left( t \right) - \lambda_{*} \right) \right\rVert ^{2} \right)	
	{+ 2} \left\langle \ddot{x} \left( t \right) + \dfrac{\alpha}{t} \dot{x} \left( t \right) , A^* \left( \lambda \left( t \right) - \lambda_{*} \right) \right\rangle \nonumber \\
	\leq & \ {\left( \theta - 1 \right) \left\lVert A^* \left( \lambda \left( t \right) - \lambda_{*} \right) \right\rVert ^{2} + \left( 2 - \dfrac{1}{2 \ell} \right) \left\lVert \nabla f \left( x \left( t \right) \right) - \nabla f \left( x_{*} \right) \right\rVert ^{2}} \nonumber \\
	& \ {+ \left( 2 \beta^{2} \left\lVert A \right\rVert ^{2} + 1 - \dfrac{\beta}{2} \right) \left\lVert Ax \left( t \right) - b \right\rVert ^{2}}\nonumber \\
\leq & \ {\left( \theta - 1 \right) \left\lVert A^* \left( \lambda \left( t \right) - \lambda_{*} \right) \right\rVert ^{2} + \Cgra \left\lVert \nabla f \left( x \left( t \right) \right) - \nabla f \left( x_{*} \right) \right\rVert ^{2}} + \Cfea \left\lVert Ax \left( t \right) - b \right\rVert ^{2}, \label{est:inq}
	\end{align}
where
$$\Cgra := \left[ 2 - \dfrac{1}{2 \ell} \right] _{+} \geq 0 \quad \mbox{and} \quad 
\Cfea  := \left[ 2 \beta^{2} \left\lVert A \right\rVert ^{2} + 1 - \dfrac{\beta}{2} \right] _{+} \geq 0. $$

	Multiplying \eqref{est:inq} by $t^{\alpha}$ and integrating, we obtain for every $t \geq t_0$
	\begin{align}
	& \ I_{1}(t) + \theta I_{2}(t) {+ \alpha I_{3}(t) + \theta I_{4}(t) + 2 I_{5}(t)} \nonumber \\
	\leq & \ {\left( \theta - 1 \right) \int_{t_{0}}^{t} s^{\alpha} \left\lVert A^* \left( \lambda \left( s \right) - \lambda_{*} \right) \right\rVert ^{2} ds + \Cgra \int_{t_{0}}^{t} s^{\alpha} \left\lVert \nabla f \left( x \left( s \right) \right) - \nabla f \left( x_{*} \right) \right\rVert ^{2} ds} \nonumber \\
 	& \ {+ \Cfea \int_{t_{0}}^{t} s^{\alpha} \left\lVert Ax \left( s \right) - b \right\rVert ^{2} ds}, \label{est:inq2}
	\end{align}
where
		\begin{align*}
		I_{1}(t)
		& := \int_{t_{0}}^{t} \left( s^{\alpha} \ddot{\varphi} \left( s \right) + \alpha s^{\alpha - 1} \dot{\varphi} \left( s \right) \right) ds, \\
		I_{2}(t)
		& := \int_{t_{0}}^{t} s^{\alpha + 1} \dot{W} \left( s \right) ds,\\
		I_{3}(t)
		& := \int_{t_{0}}^{t} s^{\alpha - 1} \left( \dfrac{d}{ds} \left\lVert \left( \dot{x} \left( s \right) , \dot{\lambda} \left( s \right) \right) \right\rVert ^{2} \right) ds,\\
		I_{4}(t)
		& := \int_{t_{0}}^{t} s^{\alpha} \left( \dfrac{d}{ds} \left( s \left\lVert A^* \left( \lambda \left( s \right) - \lambda_{*} \right) \right\rVert ^{2} \right) \right) ds, \\
		I_{5}(t)
		& := \int_{t_{0}}^{t} \left\langle s^{\alpha} \ddot{x} \left( s \right) + \alpha s^{\alpha - 1} \dot{x} \left( s \right) , A^* \left( \lambda \left( s \right) - \lambda_{*} \right) \right\rangle ds.
		\end{align*}
We will compute these five integrals separately. Let $t \geq t_0$ fixed.
\begin{itemize}
	\item
	\emph{The integral $I_{1}(t)$}.
	By the chain rule we have for all $s \in [t_0,t]$
	\begin{equation*}
	s^{\alpha} \ddot{\varphi} \left( s \right) + \alpha s^{\alpha - 1} \dot{\varphi} \left( s \right)
	= \dfrac{d}{ds} \left( s^{\alpha} \dot{\varphi} \left( s \right) \right),
	\end{equation*}
	which leads to
	\begin{equation}
	\label{int:com:I-1}
	0 = I_{1}(t) - t^{\alpha} \dot{\varphi} \left( t \right) + t_{0}^{\alpha} \dot{\varphi} \left( t_{0} \right) \leq I_{1}(t) - t^{\alpha} \dot{\varphi} \left( t \right) + t_{0}^{\alpha} |\dot{\varphi} \left( t_{0} \right)|.
	\end{equation}
	
	\item
	\emph{The integrals $I_{2}(t), I_{3}(t)$ and $I_{4}(t)$}.
Integration by parts gives
	\begin{equation*}
	I_{2} (t)= t^{\alpha + 1} W \left( t \right) - t_{0}^{\alpha + 1} W \left( t_{0} \right) - \left( \alpha + 1 \right) \int_{t_{0}}^{t} s^{\alpha} W \left( s \right) ds ,
	\end{equation*}
	which yields
	\begin{equation}
	\label{int:com:I-2}
	0 \leq t^{\alpha + 1} W \left( t \right) = I_{2}(t) + t_{0}^{\alpha + 1} W \left( t_{0} \right) + \left( \alpha + 1 \right) \int_{t_{0}}^{t} s^{\alpha} W \left( s \right) ds .
	\end{equation}
Similarly, we have
	\begin{align*}
	I_{3}(t)
	= & \ t^{\alpha - 1} \left\lVert \left( \dot{x} \left( t \right) , \dot{\lambda} \left( t \right) \right) \right\rVert ^{2} - t_{0}^{\alpha - 1} \left\lVert \left( \dot{x} \left( t_{0} \right) , \dot{\lambda} \left( t_{0} \right) \right) \right\rVert ^{2} \nonumber \\
	& \ - \left( \alpha - 1 \right) \int_{t_{0}}^{t} s^{\alpha - 2} \left\lVert \left( \dot{x} \left( s \right) , \dot{\lambda} \left( s \right) \right) \right\rVert ^{2} ds,
	\end{align*}
which yields
	\begin{align}
	0	& \leq I_{3}(t) + t_{0}^{\alpha - 1} \left\lVert \left( \dot{x} \left( t_{0} \right) , \dot{\lambda} \left( t_{0} \right) \right) \right\rVert ^{2} 
	+ \left( \alpha - 1 \right) \int_{t_{0}}^{t} s^{\alpha - 2} \left\lVert \left( \dot{x} \left( s \right) , \dot{\lambda} \left( s \right) \right) \right\rVert ^{2} ds \nonumber \\
	& \leq I_{3} (t) + t_{0}^{\alpha - 1} \left\lVert \left( \dot{x} \left( t_{0} \right) , \dot{\lambda} \left( t_{0} \right) \right) \right\rVert ^{2} 
	+ \dfrac{\alpha - 1}{t_{0}^{2}} \int_{t_{0}}^{t} s^{\alpha} \left\lVert \left( \dot{x} \left( s \right) , \dot{\lambda} \left( s \right) \right) \right\rVert ^{2} ds. \label{int:com:I-3} 
	\end{align}
Using again integration by parts, we have
	\begin{align*}
	I_{4}(t)
	& = t^{\alpha + 1} \left\lVert A^* \left( \lambda \left( t \right) - \lambda_{*} \right) \right\rVert ^{2} - t_{0}^{\alpha + 1} \left\lVert A^* \left( \lambda \left( t_{0} \right) - \lambda_{*} \right) \right\rVert ^{2} - \alpha \int_{t_{0}}^{t} s^{\alpha} \left\lVert A^* \left( \lambda \left( s \right) - \lambda_{*} \right) \right\rVert ^{2} ds \nonumber
	\end{align*}
	and from here
	\begin{equation}
	\label{int:com:I-4}
	t^{\alpha + 1} \left\lVert A^* \left( \lambda \left( t \right) - \lambda_{*} \right) \right\rVert ^{2}
	= I_{4} (t) + t_{0}^{\alpha + 1} \left\lVert A^* \left( \lambda \left( t_{0} \right) - \lambda_{*} \right) \right\rVert ^{2} + \alpha \int_{t_{0}}^{t} s^{\alpha} \left\lVert A^* \left( \lambda \left( s \right) - \lambda_{*} \right) \right\rVert ^{2} ds .
	\end{equation}
	\item
	\emph{The integral $I_{5}(t)$}.
Integration by parts gives
	\begin{align*}
	I_{5}(t)
	& = \int_{t_{0}}^{t} \left\langle \dfrac{d}{ds} \left( s^{\alpha} \dot{x} \left( s \right) \right) , A^* \left( \lambda \left( s \right) - \lambda_{*} \right) \right\rangle ds \nonumber \\
	& = t^{\alpha} \left\langle \dot{x} \left( t \right) , A^* \left( \lambda \left( t \right) - \lambda_{*} \right) \right\rangle 
	- t_{0}^{\alpha} \left\langle \dot{x} \left( t_{0} \right) , A^* \left( \lambda \left( t_{0} \right) - \lambda_{*} \right) \right\rangle 
	- \int_{t_{0}}^{t} s^{\alpha} \left\langle \dot{x} \left( s \right) , A^* \dot{\lambda} \left( s \right) \right\rangle ds. \nonumber
	\end{align*}
and, since
	\begin{align*}
	\int_{t_{0}}^{t} s^{\alpha} \left\langle \dot{x} \left( s \right) , A^* \dot{\lambda} \left( s \right) \right\rangle ds
	 \leq \dfrac{\max \left\lbrace 1 , \left\lVert A \right\rVert ^{2} \right\rbrace}{2} \int_{t_{0}}^{t} s^{\alpha} \left( \left\lVert \dot{x} \left( s \right) \right\rVert ^{2} + \left\lVert \dot{\lambda} \left( s \right) \right\rVert ^{2} \right) ds,
	\end{align*}
we obtain
	\begin{align}
	0 
	\leq & \ I_{5}(t) - t^{\alpha} \left\langle \dot{x} \left( t \right) , A^* \left( \lambda \left( t \right) - \lambda_{*} \right) \right\rangle 
	+ t_{0}^{\alpha} |\left\langle \dot{x} \left( t_{0} \right) , A^* \left( \lambda \left( t_{0} \right) - \lambda_{*} \right) \right\rangle| \nonumber \\
	& \ + \dfrac{\max \left\lbrace 1 , \left\lVert A \right\rVert ^{2} \right\rbrace}{2} \int_{t_{0}}^{t} s^{\alpha} \left\lVert \left( \dot{x} \left( s \right) , \dot{\lambda} \left( s \right) \right) \right\rVert ^{2} ds . \label{int:com:I-5}
	\end{align}
\end{itemize}
Combining \eqref{int:com:I-1}, \eqref{int:com:I-2}, \eqref{int:com:I-3}, \eqref{int:com:I-4} and \eqref{int:com:I-5}, we obtain
\begin{align}
& \ {\theta} t^{\alpha + 1} \left\lVert A^* \left( \lambda \left( t \right) - \lambda_{*} \right) \right\rVert ^{2} \nonumber\\
\leq & \ I_{1}(t) + \theta I_{2}(t) + {\alpha I_{3}(t) + \theta I_{4}(t)+ 2 I_{5}(t)}
- t^{\alpha} \dot{\varphi} \left( t \right) \nonumber \\
& \ + \int_{t_{0}}^{t} s^{\alpha} \left(\theta \left( \alpha + 1 \right) W \left( s \right) + {\left( \dfrac{\alpha \left( \alpha - 1 \right)}{t_{0}^{2}} + \max \left\lbrace 1 , \left\lVert A \right\rVert ^{2} \right\rbrace \right) \left\lVert \left( \dot{x} \left( s \right) , \dot{\lambda} \left( s \right) \right) \right\rVert ^{2}} \right) ds \nonumber \\
& \ {+ \theta \alpha} \int_{t_{0}}^{t} s^{\alpha} \left\lVert A^* \left( \lambda \left( s \right) - \lambda_{*} \right) \right\rVert ^{2} ds {- 2 t^{\alpha}} \left\langle \dot{x} \left( t \right) , A^* \left( \lambda \left( t \right) - \lambda_{*} \right) \right\rangle + \Cacc \nonumber \\
\leq & \ {- t^{\alpha} \dot{\varphi} \left( t \right) 
	+ \int_{t_{0}}^{t} s^{\alpha} V \left( s \right) ds + \left( \theta \left( \alpha + 1 \right) - 1 \right) \int_{t_{0}}^{t} s^{\alpha} \left\lVert A^* \left( \lambda \left( s \right) - \lambda_{*} \right) \right\rVert ^{2} ds} \nonumber \\
	& \ {- 2 t^{\alpha} \left\langle \dot{x} \left( t \right) , A^* \left( \lambda \left( t \right) - \lambda_{*} \right) \right\rangle + \Cacc} , \label{int:sum}
\end{align}
where {the last inequality follows from \eqref{est:inq2}},
\begin{align*}
V \left( s \right)
 := & \ \theta \left( \alpha + 1 \right) W \left( s \right) + {\left( \dfrac{\alpha \left( \alpha - 1 \right)}{t_{0}^{2}} + \max \left\lbrace 1 , \left\lVert A \right\rVert ^{2} \right\rbrace \right) \left\lVert \left( \dot{x} \left( s \right) , \dot{\lambda} \left( s \right) \right) \right\rVert ^{2}}\nonumber \\
 & \ {+ \Cgra \left\lVert \nabla f \left( x \left( s \right) \right) - \nabla f \left( x_{*} \right) \right\rVert ^{2} + \Cfea \left\lVert Ax \left( s \right) - b \right\rVert ^{2}}  \geq 0 \quad \forall s \geq t_0
\end{align*}
and
\begin{align*}
\Cacc
 := & \ t_{0}^{\alpha} \left\lvert \dot{\varphi} \left( t_{0} \right) \right\rvert
+ \theta t_{0}^{\alpha + 1} W \left( t_{0} \right) {+ \alpha t_{0}^{\alpha - 1} \left\lVert \left( \dot{x} \left( t_{0} \right) , \dot{\lambda} \left( t_{0} \right) \right) \right\rVert ^{2}} \nonumber \\
& \ {+ \theta} t_{0}^{\alpha + 1} \left\lVert A^* \left( \lambda \left( t_{0} \right) - \lambda_{*} \right) \right\rVert ^{2}
+ {2 t_{0}^{\alpha}} \left\lvert \left\langle \dot{x} \left( t_{0} \right) , A^* \left( \lambda \left( t_{0} \right) - \lambda_{*} \right) \right\rangle \right\rvert \geq 0.
\end{align*}
Dividing \eqref{int:sum} by $t^{\alpha}$ we obtain from here
\begin{align}
{\theta} t \left\lVert A^* \left( \lambda \left( t \right) - \lambda_{*} \right) \right\rVert ^{2}
\leq & - \dot{\varphi} \left( t \right) 
+ \dfrac{1}{t^{\alpha}} \int_{t_{0}}^{t} s^{\alpha} V \left( s \right) ds {+ \dfrac{\left( \theta \left( \alpha + 1 \right) - 1 \right)}{t^{\alpha}}} \int_{t_{0}}^{t} s^{\alpha} \left\lVert A^* \left( \lambda \left( s \right) - \lambda_{*} \right) \right\rVert ^{2} ds \nonumber \\
& {- 2} \left\langle \dot{x} \left( t \right) , A^* \left( \lambda \left( t \right) - \lambda_{*} \right) \right\rangle + \dfrac{\Cacc}{t^{\alpha}},\label{int:inq}
\end{align}
which holds for every $t \geq t_{0}$. We choose $r \geq t_{0}$ and integrate \eqref{int:inq} from $t_{0}$ to $r$. This yields
\begin{align}
{\theta} \int_{t_{0}}^{r} t \left\lVert A^* \left( \lambda \left( t \right) - \lambda_{*} \right) \right\rVert ^{2} dt
\leq & \ \varphi \left( t_{0} \right) - \varphi \left( r \right) 
+ \int_{t_{0}}^{r} \dfrac{1}{t^{\alpha}} \left( \int_{t_{0}}^{t} s^{\alpha} V \left( s \right) ds \right) dt \nonumber \\
& + {\left( \theta \left( \alpha + 1 \right) - 1 \right)} \int_{t_{0}}^{r} \dfrac{1}{t^{\alpha}} \left( \int_{t_{0}}^{t} s^{\alpha} \left\lVert A^* \left( \lambda \left( s \right) - \lambda_{*} \right) \right\rVert ^{2} ds \right) dt \nonumber \\
& {- 2} \int_{t_{0}}^{r} \left\langle A \dot{x} \left( t \right) , \lambda \left( t \right) - \lambda_{*} \right\rangle dt + \Cacc \int_{t_{0}}^{r} \dfrac{1}{t^{\alpha}} dt. \label{int:int}
\end{align}
Recall that
\begin{equation}
\label{int:1-t-alpha}
\int_{t_{0}}^{r} \dfrac{1}{t^{\alpha}} dt \leq \dfrac{1}{\left( \alpha - 1 \right) t_{0}^{\alpha - 1}} .
\end{equation}
Moreover, by applying Lemma \ref{lem:est-int} with $h \left( t \right) := t V \left( t \right)$, it yields
\begin{align}
\int_{t_{0}}^{r} \dfrac{1}{t^{\alpha}} \left( \int_{t_{0}}^{t} s^{\alpha} V \left( s \right) ds \right) dt
& \leq \dfrac{1}{\alpha - 1} \int_{t_{0}}^{r} t V \left( t \right) dt . \label{int:Fubini:W}
\end{align}
Similarly, applying the same result with $h \left( t \right) := t \left\lVert A^* \left( \lambda \left( t \right) - \lambda_{*} \right) \right\rVert ^{2}$ gives
\begin{equation}
\label{int:Fubini:s-lambda}
\int_{t_{0}}^{r} \dfrac{1}{t^{\alpha}} \left( \int_{t_{0}}^{t} s^{\alpha} \left\lVert A^* \left( \lambda \left( s \right) - \lambda_{*} \right) \right\rVert ^{2} ds \right) dt
\leq \dfrac{1}{\alpha - 1} \int_{t_{0}}^{r} t \left\lVert A^* \left( \lambda \left( t \right) - \lambda_{*} \right) \right\rVert ^{2} dt .
\end{equation}
Using again integration by parts we obtain
\begin{align}
& \ - \int_{t_{0}}^{r} \left\langle A \dot{x} \left( t \right) , \lambda \left( t \right) - \lambda_{*} \right\rangle dt \nonumber \\
= & \ - \left\langle Ax \left( r \right) - b , \lambda \left( r \right) - \lambda_{*} \right\rangle
+ \left\langle Ax \left( t_{0} \right) - b , \lambda \left( t_{0} \right) - \lambda_{*} \right\rangle 
+ \int_{t_{0}}^{r} \left\langle Ax \left( t \right) - b , \dot{\lambda} \left( t \right) \right\rangle dt \nonumber \\
\leq & \ \left\lVert Ax \left( r \right) - b \right\rVert \left\lVert \lambda \left( r \right) - \lambda_{*} \right\rVert + \left\lVert Ax \left( t_{0} \right) - b \right\rVert \left\lVert \lambda \left( t_{0} \right) - \lambda_{*} \right\rVert + \int_{t_{0}}^{r} \left\langle Ax \left( t \right) - b , \dot{\lambda} \left( t \right) \right\rangle dt \nonumber\\
\leq & \ \sup_{t \geq t_{0}} \left\lbrace \left\lVert Ax \left( t \right) - b \right\rVert \left\lVert \lambda \left( t \right) - \lambda_{*} \right\rVert \right\rbrace 
+ \left\lVert Ax \left( t_{0} \right) - b \right\rVert \left\lVert \lambda \left( t_{0} \right) - \lambda_{*} \right\rVert \nonumber \\
& \ + \dfrac{1}{2} \int_{t_{0}}^{r} \left( \left\lVert Ax \left( t \right) - b \right\rVert ^{2} + \left\lVert \dot{\lambda} \left( t \right) \right\rVert ^{2} \right) dt.\label{int:inner}
\end{align}
Due to the boundedness of the trajectory we have $$\sup\limits_{t \geq t_{0}} \left\lbrace \left\lVert Ax \left( t \right) - b \right\rVert \left\lVert \lambda \left( t \right) - \lambda_{*} \right\rVert \right\rbrace < + \infty.$$
Combining \eqref{int:1-t-alpha}, \eqref{int:Fubini:W}, \eqref{int:Fubini:s-lambda} and \eqref{int:inner} with  \eqref{int:int} and using the nonnegativity of $\varphi$, we obtain
\begin{align}
& {\dfrac{1 - 2 \theta}{\alpha - 1} \int_{t_{0}}^{r} t \left\lVert A^* \left( \lambda \left( t \right) - \lambda_{*} \right) \right\rVert ^{2} dt = \left( \theta + \dfrac{1 - \theta \left( \alpha + 1 \right)}{\alpha - 1} \right)} \int_{t_{0}}^{r} t \left\lVert A^* \left( \lambda \left( t \right) - \lambda_{*} \right) \right\rVert ^{2} dt \nonumber \\
\leq  & \  \dfrac{1}{\alpha - 1} \int_{t_{0}}^{r} t V \left( t \right) dt 
+ \int_{t_{0}}^{r} t \left( \left\lVert Ax \left( t \right) - b \right\rVert ^{2} + \left\lVert \dot{\lambda} \left( t \right) \right\rVert ^{2} \right) dt + \Czms \nonumber \\
\leq & \ \dfrac{1}{\alpha - 1} \int_{t_{0}}^{+ \infty} t V \left( t \right) dt 
+ \int_{t_{0}}^{+ \infty} t \left( \left\lVert Ax \left( t \right) - b \right\rVert ^{2} + \left\lVert \dot{\lambda} \left( t \right) \right\rVert ^{2} \right) dt + \Czms , \label{int:fin}
\end{align}
where
\begin{align*}
\Czms 
:= & \ \varphi \left( t_{0} \right) 
+ 2 \sup_{t \geq t_{0}} \left\lbrace \left\lVert Ax \left( t \right) - b \right\rVert \left\lVert \lambda \left( t \right) - \lambda_{*} \right\rVert \right\rbrace 
+ 2 \left\lVert Ax \left( t_{0} \right) - b \right\rVert \left\lVert \lambda \left( t_{0} \right) - \lambda_{*} \right\rVert + \dfrac{\Cacc}{\left( \alpha - 1 \right) t_{0}^{\alpha - 1}}.
\end{align*}
According to \eqref{bnd:int:Lag} and \eqref{bnd:int:vel} in Theorem \ref{thm:bnd} as well as \eqref{int:grad} and \eqref{int:fea} in Proposition \ref{prop:int}, we conclude that both $t \mapsto  tV \left( t \right)$ and $t \mapsto t \left( \left\lVert Ax \left( t \right) - b \right\rVert ^{2} + \left\lVert \dot{\lambda} \left( t \right) \right\rVert ^{2} \right)$ belong to $\sL^{1} \left(\left[ t_{0} , + \infty \right) \right)$, therefore the right-hand side of \eqref{int:fin} is finite.

Hence, by passing $r \to + \infty$ in \eqref{int:fin} and by taking into account the choice of the parameters $\theta$ and $\alpha$, we obtain the desired statement.
\end{proof}

The following result will be used to show the weak convergence of the trajectory, but it also has its own interest, since it provides the convergence rate for the KKT system associated to problem \eqref{intro:pb}.

\begin{thm}
	\label{thm:rate-KKT}
	Let $\left( x , \lambda \right) \colon \left[ t_{0} , + \infty \right) \to \sX \times \sY$ be a solution of \eqref{ds:PD-AVD} and $\left( x_{*} , \lambda_{*} \right) \in \sol$. Then it holds:
	\begin{equation}
	\label{rate-KKT:split}
	\left\lVert A^* \left( \lambda \left( t \right) - \lambda_{*} \right) \right\rVert 
	= o \left( \dfrac{1}{\sqrt{t}} \right) \quad \mbox{and} \quad \left\lVert \nabla f \left( x \left( t \right) \right) - \nabla f \left( x_{*} \right) \right\rVert
	= o \left( \dfrac{1}{\sqrt{t}} \right) \textrm{ as } t \to + \infty .
	\end{equation}
Consequently,
	\begin{align*}
	\left\lVert \nabla_{x} \Lag \bigl( x \left( t \right) , \lambda \left( t \right) \bigr) \right\rVert 
	& = \left\lVert \nabla f \left( x \left( t \right) \right) + A^* \lambda \left( t \right) \right\rVert
	= o \left( \dfrac{1}{\sqrt{t}} \right) \textrm{ as } t \to + \infty,
\end{align*}
while, as seen in Section \ref{sec:Lya},
\begin{align*}
	\left\lVert \nabla_{\lambda} \Lag \bigl( x \left( t \right) , \lambda \left( t \right) \bigr) \right\rVert 
	& = \left\lVert Ax \left( t \right) - b \right\rVert = \bO \left( \dfrac{1}{t^{2}} \right) \textrm{ as } t \to + \infty .
	\end{align*}
\end{thm}
\begin{proof}
The continously differentiable functions
	\begin{align*}
	F \left( t \right)
	& := t \left\lVert A^* \left( \lambda \left( t \right) - \lambda_{*} \right) \right\rVert ^{2} \geq 0 \nonumber \\
	G \left( t \right)
	& := \left( 1 + t \left\lVert A \right\rVert ^{2} \right) \left\lVert A^* \left( \lambda \left( t \right) - \lambda_{*} \right) \right\rVert ^{2} + t \left\lVert \dot{\lambda} \left( t \right) \right\rVert ^{2}
	\end{align*}
defined on $[t_0,+\infty)$ belong, according to Proposition \ref{prop:int:t-lambda} and Theorem \ref{thm:bnd}, to 
$\sL^{1} \left([t_{0} , + \infty) \right)$. For every $t \geq t_{0}$ we have
	\begin{align*}
	\dfrac{d}{dt} \left( t \left\lVert A^* \left( \lambda \left( t \right) - \lambda_{*} \right) \right\rVert ^{2} \right)
	& = \left\lVert A^* \left( \lambda \left( t \right) - \lambda_{*} \right) \right\rVert ^{2}
	+ 2t \left\langle AA^* \left( \lambda \left( t \right) - \lambda_{*} \right) , \dot{\lambda} \left( t \right) \right\rangle \nonumber \\
	& \leq \left\lVert A^* \left( \lambda \left( t \right) - \lambda_{*} \right) \right\rVert ^{2}
	+ t \left( \left\lVert AA^* \left( \lambda \left( t \right) - \lambda_{*} \right) \right\rVert ^{2} + \left\lVert \dot{\lambda} \left( t \right) \right\rVert ^{2} \right) \nonumber \\
	& \leq \left( 1 + t \left\lVert A \right\rVert ^{2} \right)\left\lVert A^* \left( \lambda \left( t \right) - \lambda_{*} \right) \right\rVert ^{2}+ t \left\lVert \dot{\lambda} \left( t \right) \right\rVert ^{2},
	\end{align*}
thus, from Lemma \ref{lem:lim-0} we get
	\begin{equation}
	\label{rate-KKT:A-lambda}
	\left\lVert A^* \left( \lambda \left( t \right) - \lambda_{*} \right) \right\rVert = o \left( \dfrac{1}{\sqrt{t}} \right) \quad \textrm{as} \quad t \to + \infty.
	\end{equation}
The functions
	\begin{align*}
	F \left( t \right)
	& := t \left\lVert \nabla f \left( x \left( t \right) \right) - \nabla f \left( x_{*} \right) \right\rVert ^{2} \geq 0 \nonumber \\
	G \left( t \right)
	& := \left( 1 + t \right) \left\lVert \nabla f \left( x \left( t \right) \right) - \nabla f \left( x_{*} \right) \right\rVert ^{2} + t \ell^{2} \left\lVert \dot{x} \left( t \right) \right\rVert ^{2}
	\end{align*}
defined on $[t_0,+\infty)$ are locally absolutely continuous and belong, according to Proposition \ref{prop:int} and Theorem \ref{thm:bnd}, to  $\sL^{1} \left([t_{0} , + \infty) \right)$. For almost every $t \geq t_{0}$ we have
	\begin{align*}
	& \dfrac{d}{dt} \left( t \left\lVert \nabla f \left( x \left( t \right) \right) - \nabla f \left( x_{*} \right) \right\rVert ^{2} \right) \nonumber \\
	= \ 	& \left\lVert \nabla f \left( x \left( t \right) \right) - \nabla f \left( x_{*} \right) \right\rVert ^{2}
	+ 2t \left\langle \nabla f \left( x \left( t \right) \right) - \nabla f \left( x_{*} \right) , \dfrac{d}{dt} \nabla f \left( x \left( t \right) \right) \right\rangle \nonumber \\
	\leq \ 	& \left( 1 + t \right) \left\lVert \nabla f \left( x \left( t \right) \right) - \nabla f \left( x_{*} \right) \right\rVert ^{2} + t \left\lVert \dfrac{d}{dt} \nabla f \left( x \left( t \right) \right) \right\rVert ^{2} \nonumber \\
	\leq \ 	& \left( 1 + t \right) \left\lVert \nabla f \left( x \left( t \right) \right) - \nabla f \left( x_{*} \right) \right\rVert ^{2} + t \ell^{2} \left\lVert \dot{x} \left( t \right) \right\rVert ^{2} ,
	\end{align*}
	where the last inequality follows from the fact that $\nabla f$ is $\ell-$Lipschitz continuous. From Lemma \ref{lem:lim-0} we get
	\begin{equation*}
	\left\lVert \nabla f \left( x \left( t \right) \right) - \nabla f \left( x_{*} \right) \right\rVert = o \left( \dfrac{1}{
	\sqrt{t}} \right) \quad \textrm{as} \quad t \to + \infty.
	\end{equation*}
According to \eqref{rate-KKT:A-lambda} we have
	\begin{align*}
	\left\lVert \nabla_{x} \Lag \bigl( x \left( t \right) , \lambda \left( t \right) \bigr) \right\rVert 
	& = \left\lVert \nabla f \left( x \left( t \right) \right) + A^* \lambda \left( t \right) \right\rVert \nonumber \\
	& \leq \left\lVert \nabla f \left( x \left( t \right) \right) - \nabla f \left( x_{*} \right) \right\rVert + \left\lVert A^* \left( \lambda \left( t \right) - \lambda_{*} \right) \right\rVert
	= o \left( \dfrac{1}{\sqrt{t}} \right) \quad \textrm{as} \quad t \to + \infty,
	\end{align*}
while Theorem \ref{thm:rate} gives
	\begin{equation*}
	\left\lVert \nabla_{\lambda} \Lag \bigl( x \left( t \right) , \lambda \left( t \right) \bigr) \right\rVert = \left\lVert Ax \left( t \right) - b \right\rVert = \bO \left( \dfrac{1}{t^{2}} \right) \quad \textrm{as} \quad t \to + \infty.
	\end{equation*}
\end{proof}

We are now in the position to prove the main result of this section.
\begin{thm}\label{weakconvergence}
Let $\left( x , \lambda \right) \colon \left[ t_{0} , + \infty \right) \to \sX \times \sY$ be a solution of \eqref{ds:PD-AVD} and $\left( x_{*} , \lambda_{*} \right) \in \sol$. Then $\bigl( x \left( t \right) , \lambda \left( t \right) \bigr)$ converges weakly to a primal-dual optimal solution of \eqref{intro:pb} as $t \to + \infty$.
\end{thm}
\begin{proof}
We have seen in Lemma \ref{lem:lim-phi} that the limit $\lim\limits_{t \to + \infty} \left\lVert \bigl( x \left( t \right) , \lambda \left( t \right) \bigr) - \left( x_{*} , \lambda_{*} \right) \right\rVert$ exists for every $\left( x_{*} , \lambda_{*} \right) \in \sol$, which proves condition (i) of  Opial's Lemma (see Lemma \ref{lem:Opial}). 

In order to prove condition (ii), we consider $\left( \tx , \tlambda \right)$ an arbitrary weak sequential cluster point of $\bigl( x \left( t \right) , \lambda \left( t \right) \bigr)$ as $t \rightarrow +\infty$, which means that there exists a sequence $\left\lbrace \left( x \left( t_{n} \right) , \lambda \left( t_{n} \right) \right) \right\rbrace _{n \geq 0}$ such that
	\begin{equation*}
	\left( x \left( t_{n} \right) , \lambda \left( t_{n} \right) \right) \rightharpoonup \left( \tx , \tlambda \right) \quad \textrm{as} \ n \to + \infty .
	\end{equation*}
		
	Theorem \ref{thm:rate-KKT} and Theorem \ref{thm:rate} allow us to deduce that
	\begin{equation*}
	\nabla f \left( x \left( t_{n} \right) \right) + A^* \lambda \left( t_{n} \right) 
	\to \nabla f \left( x_{*} \right) + A^* \lambda_{*} = 0 
	\quad \textrm{as} \quad n \to + \infty .
	\end{equation*}
	and
	\begin{equation*}
	A x \left( t_{n} \right) - b \to 0 \quad \textrm{as} \quad n \to + \infty,
	\end{equation*}
respectively. Since the graph of the operator $\TL$ introduced in \eqref{TL} is sequentially closed in $\left( \sX \times \sY \right) ^{\mathrm{weak}} \times \left( \sX \times \sY \right) ^{\mathrm{strong}}$ (cf. \cite[Proposition 20.38]{Bauschke-Combettes:book}), we have that
	\begin{equation*}
	\begin{cases}
	\nabla f \left( \tx \right) + A^* \tlambda & = \nabla f \left( x_{*} \right) + A^* \lambda_{*} = 0 \\
	A \tx - b & = Ax_{*} - b = 0
	\end{cases}.
	\end{equation*}
	In other words, $\left( \tx , \tlambda \right)$ belongs to $\sol$ and the proof is complete.
\end{proof}

\begin{rmk}
In case $A := 0$ and $b := 0$, the optimization problem \eqref{intro:pb} reduces to the unconstrained optimization problem
	\begin{equation}
	\label{intro:pb-un}
	\min\limits_{x \in \sX} f \left( x \right) .
	\end{equation}
We will prove that inwe obtain as particular case all convergence results stated in the literature for Nesterov's accelerated gradient system \eqref{ds:AVD}.

Indeed, the system of optimality conditions \eqref{intro:opt-Lag} read in this case
	\begin{equation*}
	\left( x_{*} , \lambda_{*} \right) \in \sol
	\Leftrightarrow \nabla f \left( x_{*} \right) = 0 \textrm{ and } \lambda_{*} \in \sY,
	\end{equation*}
in particular, $x_{*} \in \sX$ is an optimal solution of \eqref{intro:pb-un} if and only if $\nabla f \left( x_{*} \right) = 0$.
	The system \eqref{ds:PD-AVD} becomes
	\begin{equation*}
	\begin{dcases}
	\ddot{x} \left( t \right) + \dfrac{\alpha}{t} \dot{x} \left( t \right) + \nabla f \left( x \left( t \right) \right)    		& = 0 \\
	\ddot{\lambda} \left( t \right) + \dfrac{\alpha}{t} \dot{\lambda} \left( t \right)	& = 0 \\
\Bigl( x \left( t_{0} \right) , \lambda \left( t_{0} \right) \Bigr) 			= \Bigl( x_{0} , \lambda_{0} \Bigr) \textrm{ and }
\Bigl( \dot{x} \left( t_{0} \right) , \dot{\lambda} \left( t_{0} \right) \Bigr) = \Bigl( \dot{x}_{0} , \dot{\lambda}_{0} \Bigr)
	\end{dcases}.
	\end{equation*}
The dynamical system in $x$ is reads
	\begin{equation*}
	\begin{dcases}
	\ddot{x} \left( t \right) + \dfrac{\alpha}{t} \dot{x} \left( t \right) + \nabla f \left( x \left( t \right) \right)   = 0 \\
x(t_0) = x_0 \ \textrm{ and } \dot{x}(t_0) = \dot{x}_0
	\end{dcases},
	\end{equation*}
for $\alpha \geq 3$, and is nothing else than Nesterov's accelerated gradient system. The trajectory generated by the system in $\lambda$ is $\lambda(t) = \frac{\dot{\lambda}_0 t_0^\alpha}{1-\alpha} t^{1-\alpha} + \lambda_0 - \frac{\dot{\lambda}_0 t_0}{1-\alpha}$ for every $t \geq t_{0}$. The parameters $\beta$ and $\theta$ play no role in the system.
	
If $\alpha \geq 3$, then Theorem \ref{thm:rate} (ii) gives that $f(x(t))$ converges to $f_*$ with a rate of convergence of $\bO \left( \dfrac{1}{t^{2}} \right)$ as $t \to + \infty$, which is the rate reported in \cite{Attouch-Chbani-Peypouquet-Redont,Su-Boyd-Candes} for \eqref{ds:AVD}.

If $\alpha > 3$, then Theorem \ref{weakconvergence} gives that the trajectory $x(t)$ converges weakly to an optimal solution of \eqref{intro:pb-un}, as $t \to + \infty$, which agrees with what it has been reported in \cite{Attouch-Chbani-Peypouquet-Redont} for \eqref{ds:AVD}.

Finally, we mention that the convergence of the trajectory in the critical case $\alpha = 3$ (\cite{Attouch-Chbani-Peypouquet-Redont,Su-Boyd-Candes}) is still an open question, as it is the convergence of the iterates of the original Nesterov's acceleration algorithm (\cite{Attouch-Cabot:18,FISTA,Nesterov:83}).
\end{rmk}

\appendix

\section{Appendix}

We collect here some results which are used in the proof of the convergence of the trajectory of the dynamical system \eqref{ds:PD-AVD}.

\begin{lem}
	\label{lem:est-int}
	Let $0 < \delta \leq r \leq + \infty$ and $h \colon [\delta, +\infty) \to [0,+\infty)$ be a continuous function.  For every $\alpha > 1$ it holds
	\begin{equation*}
	\int_{\delta}^{r} \dfrac{1}{t^{\alpha}} \left( \int_{\delta}^{t} s^{\alpha - 1} h \left( s \right) ds \right) dt \leq \dfrac{1}{\alpha - 1} \int_{\delta}^{r} h \left( t \right) dt .
	\end{equation*}
If $r = + \infty$, then equality holds.
\end{lem}
\begin{proof}
	We have
	\begin{align*}
	\int_{\delta}^{r} \dfrac{1}{t^{\alpha}} \left( \int_{\delta}^{t} s^{\alpha - 1} h \left( s \right) ds \right) dt
	& = \int_{\delta}^{r} \int_{\delta}^{t} \dfrac{1}{t^{\alpha}}  s^{\alpha - 1} h \left( s \right) ds dt
	= \iint_{\mathcal{A}} \dfrac{1}{t^{\alpha}} s^{\alpha - 1} h \left( s \right) d \mathcal{A} ,
	\end{align*}
	where
	\begin{align*}
	\mathcal{A} 
	:= \left\lbrace \left( s , t \right) \colon \delta \leq t \leq r , \delta \leq s \leq t \right\rbrace
	= \left\lbrace \left( t , s \right) \colon \delta \leq s \leq r , s \leq t \leq r \right\rbrace .
	\end{align*}
	Thus, by applying Fubini's theorem,
	\begin{align*}
	\int_{\delta}^{r} \dfrac{1}{t^{\alpha}} \left( \int_{\delta}^{t} s^{\alpha - 1} h \left( s \right) ds \right) dt
	& = \int_{\delta}^{r} \int_{s}^{r} \dfrac{1}{t^{\alpha}} s^{\alpha - 1} h \left( s \right) dt ds
	= \int_{\delta}^{r} s^{\alpha - 1} h \left( s \right) \left( \int_{s}^{r} \dfrac{1}{t^{\alpha}} dt \right) ds ,
	\end{align*}
	from which we get the desired estimate, as
	\begin{equation*}
	\int_{s}^{r} \dfrac{1}{t^{\alpha}} dt = \dfrac{1}{\alpha - 1} \left( \dfrac{1}{s^{\alpha - 1}} - \dfrac{1}{r^{\alpha - 1}} \right) \leq \dfrac{1}{\left( \alpha - 1 \right) s^{\alpha - 1}} .
	\end{equation*}
	If $r := + \infty,$ then the above inequality is an equality.
\end{proof}

The following result can be found in \cite[Lemma 5.2]{Abbas-Attouch-Svaiter}.
\begin{lem}
	\label{lem:lim-0}
	Let $\delta > 0$, $1 \leq p < \infty$ and $1 \leq q \leq \infty$.
	Suppose that $F \in \sL^{p} \left( \left[ \delta , + \infty \right) \right)$ is a locally absolutely continuous nonnegative function, $G \in \sL^{q} \left( \left[ \delta , + \infty \right) \right)$ and
	\begin{equation*}
	\dfrac{d}{dt} F \left( t \right) \leq G \left( t \right) \quad \mbox{for almost every} \quad t \geq \delta.
	\end{equation*}
Then $\lim\limits_{t \to + \infty} F \left( t \right) = 0$.
\end{lem}

Opial’s Lemma \cite{Opial} in continuous form is used in the proof of the weak convergence of the trajectory of \eqref{ds:PD-AVD} to a primal-dual solution of \eqref{intro:pb}. This argument was first used in \cite{Bruck} to establish the convergence of nonlinear contraction semigroups.
\begin{lem}
	\label{lem:Opial}
	Let $S$ be a nonempty subset of $\sX$ and $z \colon \left[ t_{0} , + \infty \right) \to \sX$.
	Assume that
	\begin{enumerate}
		\item 
		for every $z_{*} \in S$, $\lim\limits_{t \to + \infty} \left\lVert z \left( t \right) - z_{*} \right\rVert$ exists;
		
		\item 
		every weak sequential cluster point of the trajectory $z \left( t \right)$ as $t \to + \infty$ belongs to $S$.
	\end{enumerate}
Then $z$ converges weakly to a point in $S$ as $t \to + \infty$.
\end{lem}
	
Statement \eqref{rate-KKT:split} in Theorem \ref{thm:rate-KKT} suggests that the mapping $(x,\lambda) \mapsto (\nabla f(x), A^*\lambda)$ is constant along the set $\sol$ of primal-dual optimal solutions of \eqref{intro:pb}. This is confirmed by the following result.
\begin{prop}
Consider the optimization problem \eqref{intro:pb}. If $\nabla f$ is $\ell-$Lipschitz continuous, 
then for every $\left( x_{*} , \lambda_{*} \right), \left( x_{**} , \lambda_{**} \right) \in \sol$ it holds
$$\nabla f \left( x_{*} \right) = \nabla f \left( x_{**} \right) \quad \mbox{and} \quad A^*\lambda_* =  A^* \lambda_{**}.$$
\end{prop}
\begin{proof}
	Let $\left( x_{*} , \lambda_{*} \right), \left( x_{**} , \lambda_{**} \right) \in \sol$. We have $Ax_{*} = Ax_{**} = b$.
According to the Baillon-Haddad theorem \cite[Corollary 18.17]{Bauschke-Combettes:book}, $\nabla f$ is $\ell^{-1}$-cocoercive, which means
	\begin{align*}
	\dfrac{1}{\ell} \left\lVert \nabla f \left( x_{**} \right) - \nabla f \left( x_{*} \right) \right\rVert ^{2}
	& \leq \left\langle \nabla f \left( x_{**} \right) - \nabla f \left( x_{*} \right) , x_{**} - x_{*} \right\rangle \nonumber \\
	& = - \left\langle A^* \left( \lambda_{**} - \lambda_{*} \right) , x_{**} - x_{*} \right\rangle  =  \left\langle \lambda_{*} - \lambda_{**}, Ax_{**} - Ax_{*} \right\rangle  = 0 ,
	\end{align*}
	where the first equation comes from \eqref{intro:opt-Lag}. This yields $\nabla f \left( x_{**} \right) = \nabla f \left( x_{*} \right)$, which, again via  \eqref{intro:opt-Lag},  gives $A^*\lambda_* =  A^* \lambda_{**}$.
\end{proof}

{\bf Acknowledgements.} The authors are thankful to Ern\"o Robert Csetnek (University of Vienna) for comments and remarks which have improved the quality of the paper.
	
%
%
%
%


\end{document}